\documentclass[sn-mathphys,Numbered, icolw]{sn-jnl}

\usepackage{graphicx}%
\usepackage{multirow}%
\usepackage{amsmath,amssymb,amsfonts}%
\usepackage{amsthm}
\usepackage{mathrsfs}%
\usepackage[title]{appendix}%
\usepackage{xcolor}%
\usepackage{nicefrac,xfrac}
\usepackage{textcomp}%
\usepackage{manyfoot}%
\usepackage{booktabs}%
\usepackage{algorithm}%
\usepackage{algorithmicx}%
\usepackage{algpseudocode}%
\usepackage{listings}%
\usepackage{tikz}

\theoremstyle{thmstyleone}%
\newtheorem{theorem}{Theorem}
\theoremstyle{thmstyletwo}%
\newtheorem{remark}{Remark}%

\theoremstyle{thmstylethree}%

\usepackage[utf8]{inputenc}
\usepackage{amssymb,amsmath,amsthm}
\usepackage{mathtools,cases,csquotes}
\usepackage{cancel}
\usepackage{amsbsy,csquotes}
\usepackage{booktabs}
\usepackage{subcaption}
\usepackage{slashbox}
\usepackage{amsmath, amssymb}
\usepackage{listings}
\usepackage{empheq}
\usepackage{afterpage}
\usepackage{xcolor}
\usepackage{cleveref}
\usepackage{url}
\usepackage{booktabs}
\usepackage{stmaryrd}
\usepackage{color}
\usepackage{graphicx}
\usepackage{overpic}
\usepackage{subcaption}

\usepackage{algpseudocode}
\usepackage{algorithm}
\usepackage{lipsum}
\usepackage{ulem}
\usepackage{mathtools}

\DeclareMathOperator{\DivG}{\nabla_{\mathcal{G}} \cdot}
\DeclareMathOperator{\GradG}{\nabla_{\mathcal{G}}}

\newcommand{\R}{\mathbb{R}}

\newcommand{\dtheta}{\, \mathrm{d} \theta}

\newcommand{\ds}{\, \mathrm{d} s}
\newcommand{\dr}{\, \mathrm{d} r}

\newcommand{\E}{\mathcal{E}}
\newcommand{\V}{\mathcal{V}}

\newcommand{\p}{p}
\newcommand{\q}{q}

\newcommand{\pref}{p^{\text{ref}}}
\newcommand{\vvref}{\mathbf{v}^{\text{ref}}}

\newcommand{\NN}{\mathbf{N}}
\newcommand{\BB}{\mathbf{B}}
\newcommand{\TT}{\mathbf{T}}

\newcommand{\ww}{\mathbf{w}}

\newcommand{\vv}{\mathbf{v}}
\newcommand{\nn}{\mathbf{n}}

\newcommand{\dG}{\, \mathrm{d}\mathcal{G}}
\newcommand{\dE}{\, \mathrm{d}\mathcal{E}}
\newcommand{\dV}{\, \mathrm{d}\mathcal{V}}

\newcommand{\jump}[1]{\ensuremath{[\![#1]\!]} }

\newcommand{\pdel}[1]{\partial_{#1}}

\newcommand{\pvs}[1]{pv}
\newcommand{\norm}[2]{\Vert #1 \Vert_{#2}}

\newtheorem{thm}{Theorem}[section]
\newtheorem{lem}{Lemma}[section]

\usepackage{tikz}
\usetikzlibrary{backgrounds,patterns, shapes,calc,arrows,fit,shadows,arrows,positioning,decorations.text,mindmap}

\usepackage{todonotes}
\begin{document}
\title{Directional flow in perivascular networks: Mixed finite elements for reduced-dimensional models on graphs}

\author*[1]{\fnm{Ingeborg G.} \sur{Gjerde}}\email{ingeborg@simula.no}

\author[1]{\fnm{Miroslav} \sur{Kuchta}}\email{miroslav@simula.no}

\author[1]{\fnm{Marie E.} \sur{Rognes}}\email{meg@simula.no}

\author[2]{\fnm{Barbara \sur{Wohlmuth}}}\email{wohlmuth@ma.tum.de}

\affil[1]{\orgname{Simula Research Laboratory, Oslo, Norway}}

\affil[2]{\orgname{Technical University of Munich, Germany}}
\date{\today}

\abstract{
The flow of cerebrospinal fluid through the perivascular spaces of the brain is believed to play a crucial role in eliminating toxic waste proteins. While the driving forces of this flow have been enigmatic, experiments have shown that arterial wall motion is central. In this work, we present a network model for simulating pulsatile fluid flow in perivascular networks. We establish the well-posedness of this model in the primal and dual mixed variational settings, and show how it can be discretized using mixed finite elements. Further, we utilize this model to investigate fundamental questions concerning the physical mechanisms governing perivascular fluid flow. Notably, our findings reveal that arterial pulsations can induce directional flow in branching perivascular networks.}

\maketitle

\newpage
\section{Introduction}
\label{sec:intro}

Cerebrospinal fluid flow and transport in perivascular spaces is thought to play a key role for solute influx and metabolite clearance in the brain~\cite{iliff2012paravascular, bohr2022glymphatic}. Perivascular spaces are compartments surrounding blood vessels on the brain surface and within the brain itself with the vascular wall as their inner boundary. There is substantial interest in these processes due to their association with neurodegenerative diseases such as Alzheimer's or Parkinson's diseases~\cite{tarasoff2015clearance}. Perivascular flow and transport has been linked to the cardiac rhythm~\cite{iliff2013cerebral, mestre2018flow, bedussi2018paravascular, rennels1990rapid} and to other vasomotion patterns~\cite{van2020vasomotion, munting2023spontaneous, bojarskaite2023sleep}. However, our understanding of the drivers and directionality of these flows remains incomplete. 
    
In recent years, computational modelling has emerged as a new approach for studying cerebrospinal fluid flow and transport in and around the brain. Efforts to this end included the development of hydraulic network models, with an emphasis on perivascular resistance \cite{tithof2019hydraulic, faghih2018bulk, boster2022sensitivity}. These steady-state Poiseuille flow-type models typically use pressure gradients to drive fluid flow. In contrast, flow driven by cardiac-induced arterial pulsations or vasomotion calls for pulsatile network models~\cite{rey2018pulsatile, daversin2022geom}. Additionally, vessel wall movement in the associated asymmetric domains can give rise to complex flow patterns~\cite{carr2021peristaltic, daversin2020mechanisms}. Thus, care is required both for the derivation and numerical discretization of reduced network models.  

In this work, we introduce and study network models of pulsatile flow in perivascular spaces from both a mathematical and a biological perspective. Our main questions and findings are therefore two-fold. Mathematically, we rigorously derive a graph-based model for fluid flow in perivascular networks that incorporates pulsatile wall motion, axial velocity gradients, perivascular porosity, and bifurcation conditions. The resulting Stokes--Brinkman-type system of equations takes the form of a saddle-point problem describing the perivascular cross-section fluxes and pressures. Key questions relate to the existence and uniqueness of solutions and approximations to these equations, and in particular to the approximation properties as the complexity and cardinality of the network increases. Here, we analyze and compare several continuous and discrete formulations, and demonstrate stability and robustness with respect to appropriately weighted norms. The discretized models have a low computational cost, making it straightforward to simulate perivascular flow in large networks. 

Building on this, we simulate perivascular fluid flow due to arterial wall motion in synthetic and image-based vascular networks. Our simulations give rise to the following insights 
\begin{itemize}
\setlength{\itemsep}{0em}
     \item \textit{Network branching and heterogeneity can induce directional net flow.} Modelling the perivascular spaces as annular spaces surrounding an arterial tree, with an open inlet and open outlets, we find that spatially-uniform pulsations of the vascular wall can drive directional flow. Such flow is not expected nor observed for non-bifurcating vessels. The volume of net flow increases with the number of generations in the network and persists for arbitrary wave frequencies. These observations suggest that the network complexity contributes to directional perivascular flow. 
   \item \textit{Continuous, closed perivascular networks are not conducive to net flow. Directional flow is contingent on perivascular-tissue connections.} Modelling the perivascular spaces as a continuous, closed network extending from arteries, to capillaries, to veins yields negligible net flow. The high resistance of the capillary bed and lack of other fluid inlet and outlets, seems to disrupt the interaction between the evolving pressure and resistance fields, and hence inhibits net flow generation. 
\end{itemize}

The article is structured as follows. In Section~\ref{sec:model}, we introduce geometrical and physiological assumptions and present a pulsatile perivascular network flow model. The detailed derivation of the network model from the full three-dimensional model is given in Appendix~\ref{sec:derivation}. The mathematical analysis is contained in Section~\ref{sec:graph}. Here, we formulate primal and dual mixed formulations of the network model and show well-posedness. We also introduce primal and dual finite element approximations and demonstrate uniform stability with respect to the problem parameters and network topology. We apply these methods to study perivascular pumping in physiological perivascular networks in Section \ref{sec:pvs-results}, and discuss these findings in Section \ref{sec:discussion}.

\section{Modelling pulsatile fluid flow in porous perivascular networks}
\label{sec:model}

This section introduces a rigorously-derived network model describing pulsatile flow of an incompressible fluid in porous perivascular spaces. The geometry, the original 3D model equations, the (1D) network model and its underlying assumptions, and a discussion of the resistance in relation to perivascular shape is included. The derivation of the network model is further detailed in~Appendix \ref{sec:derivation}.

\subsection{Perivascular geometry}
\label{sec:model-geom}

We model the PVS as a network of flow channels (\emph{branches}) surrounding the vasculature, described by a graph $\mathcal{G}$ representing the (peri)vascular centerlines and the cross-section shapes of the PVS (\Cref{fig:domain}). The graph is kept fixed in time, while the cross-section shape is dynamic, allowing fluid to be pushed in, through, and out of the PVS.

\subsubsection{Perivascular network structure}
\label{sec:geom-network}

The network itself is represented by an oriented spatial graph $\mathcal{G} = (\V, \E)$, with $m$ (graph) vertices $\V =\{\vv_1, ..., \vv_m \}$ for $\vv_j \in \R^3$ ($j = 1, \dots, m$) and $n$ (graph) edges $\E = \{\Lambda_1, \Lambda_2, ..., \Lambda_n \}$ connecting these vertices. We define each edge $\Lambda_i = \{\boldsymbol{\lambda}_i(s) \} \subset \mathbb{R}^3$ ($i = 1, \dots, n$) as a $C^2$-regular curve parameterized by $\boldsymbol{\lambda}_i(s)$ for $s \in (0, \ell_i)$; letting $\vert \boldsymbol{\lambda}_i'(s) \vert=1$ ($\vert \cdot\vert$ being the Euclidean norm) so that $s$ coincides with the arc-length of the curve, and $\ell_i > 0$ denotes the edge length. Moreover, if $\Lambda_i$ connects from $\vv_j$ to $\vv_k$, $\boldsymbol{\lambda}_i(0) = \vv_j$, and $\boldsymbol{\lambda}_i(\ell_i) = \vv_k$. For each vertex $\vv_j \in \V$, we denote by $E(\vv_j)$ the set of edges connected to $\vv_j$, and by $E_\text{in}(\vv_j)$ and $E_\text{out}(\vv_j)$ the edges going into and out of vertex $\vv_j$, respectively. Note that the domains $\Lambda_i$ are open, meaning that $\mathcal{E}$ and $\mathcal{V}$ are disjoint. We denote by $\Lambda \subset \R^3$ the extension of the graph,
%
%
\begin{equation*}
  \Lambda = \E \cup \V,
\end{equation*}
which is the geometric domain containing all edges and vertices. The set of vertices $\mathcal{V}$ is split into \textit{internal vertices} $\mathcal{I}$ and \textit{boundary vertices} $\partial \mathcal{V}$. By definition, each internal vertex is connected to two or more edges, while each boundary vertex is connected to a single edge.

\begin{figure}
 \vspace{5em}
     \begin{subfigure}{0.3\textwidth}
\begin{overpic}[width=.75\textwidth]{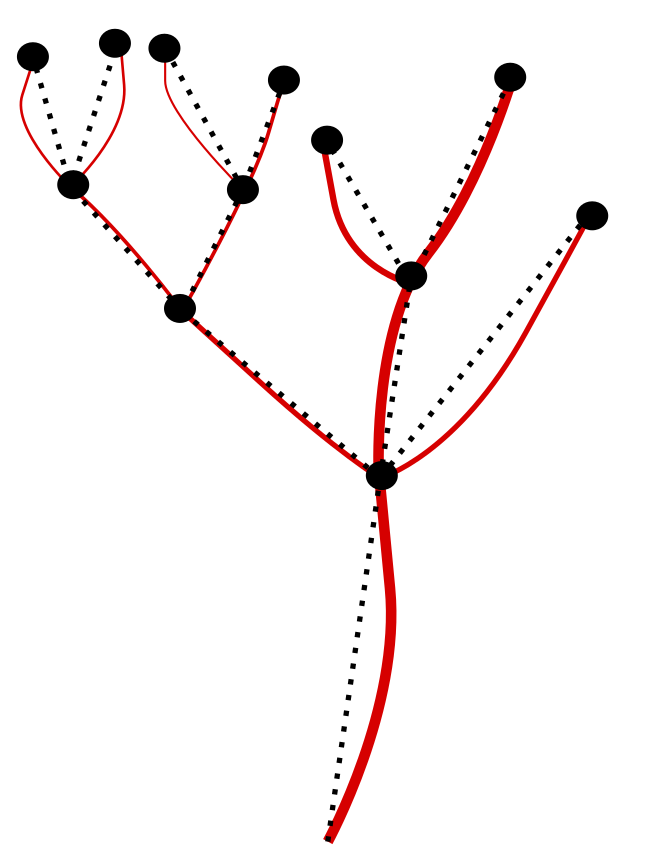}
\put(42, 3){$\textcolor{black}{\vv_j \in \partial \V}$}
\put(46, 20){$\textcolor{black}{\Lambda_i \in \mathcal{E}}$}
\put(48, 40){$\textcolor{black}{\vv_k \in \mathcal{I}}$}
\end{overpic}
    \caption{Graph structure of flow channels}
    \end{subfigure}
  \hspace*{1em}
     \begin{subfigure}{0.25\textwidth}
\begin{overpic}[width=.7\textwidth]{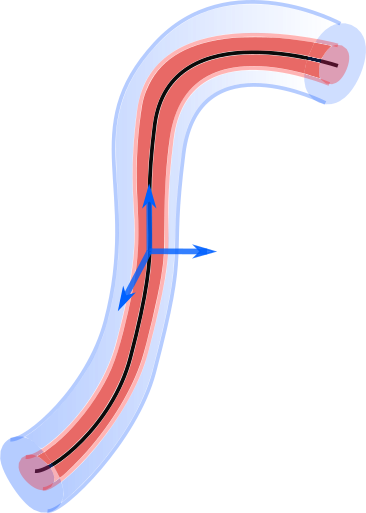}
\put(35,60){$\textcolor{blue}{\TT}$}
\put(43,48){$\textcolor{blue}{\NN}$}
\put(14,40){$\textcolor{blue}{\BB}$}
\put(28,10){$\textcolor{black}{\Lambda_i}$}
\end{overpic}
 \label{fig:domain-full}
 \caption{PVS channel geometry}
 \end{subfigure}
  \begin{subfigure}{0.3\textwidth}
   \begin{overpic}[width=.85\textwidth]{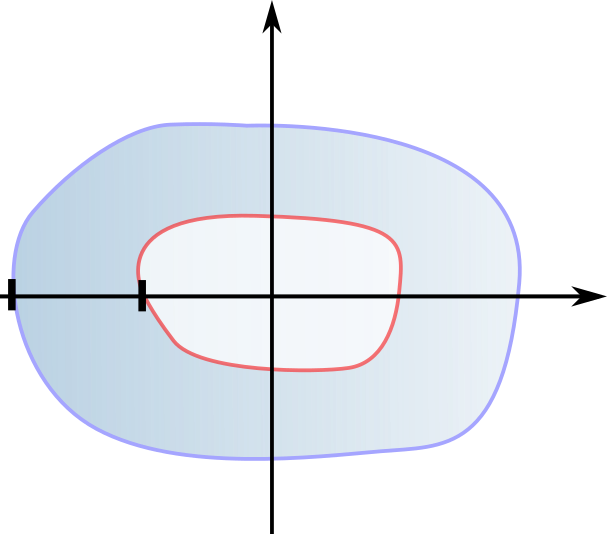}
\put(82,52){$\textcolor{blue}{\Gamma_2}$}
\put(60,52){$\textcolor{red}{\Gamma_1}$}
\put(25,28){$R^1$}
\put(0,28){$\textcolor{black}{R^2}$}
\put(70,5){$\textcolor{black}{C}$}
\end{overpic}
  \caption{Channel cross-section shape}
    \end{subfigure}          
 \caption{The perivascular space consists of a network of annular flow channels surrounding the vasculature. We organize the channels as a graph $\mathcal{G}=(\mathcal{V}, \mathcal{E})$, with internal vertices $\mathcal{I}$ and boundary vertices $\partial \V$. Each branch $i$ of the network consists of annular generalized cylinders with centerline $\Lambda_i$. The channel cross-sections $C_i(s,t)$ are characterized by a given inner radius $R^1(s,\theta,t)$ and outer radius $R^2(s,\theta,t)$.}
 \label{fig:domain}
\end{figure}
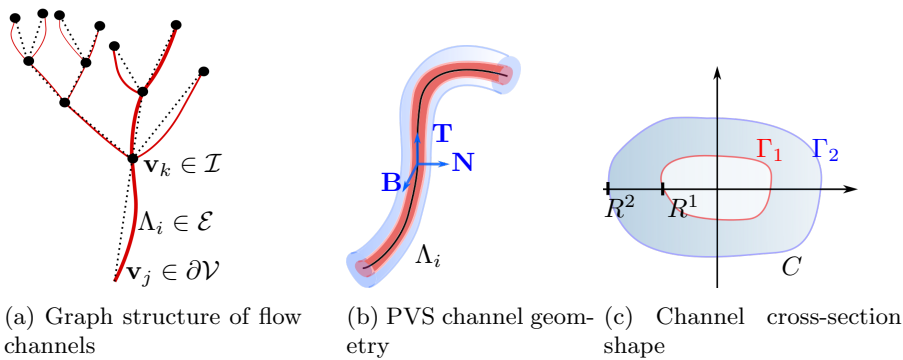

\subsubsection{Perivascular channels and cross-section shape}
\label{sec:channel}

Equipped with the graph representation of the perivascular network, we now define the PVS itself by also describing the cross-sections $C$ along the
centerlines $\boldsymbol{\lambda}$ via a suitable polar coordinate system (the Frenet-Serrat frame~\cite{generalized-cylinders}). Motivated by the physiology at hand, we assume that the cross-sections of each PVS branch can be well represented by a \textit{generalized annular} or \textit{annular-like} domain described by inner and outer curves, representing the interface towards the blood vessel and surrounding tissue, respectively (also see~\Cref{fig:domain}).

More precisely, consider first a single branch with centerline $\Lambda_i$ and its Frenet-Serrat frame $\mathbf{T}_i, \mathbf{N}_i, \mathbf{B}_i$ (representing the tangent, normal and binormal directions). We then define the channel $\Omega_i = \Omega_i(t)$ by the open domain
\begin{equation}
  \Omega_i = \{ \boldsymbol{\lambda_i}(s)  + r \cos(\theta) \NN_i(s) + r\sin(\theta) \BB_i(s), 0 < s < \ell_i, 0 < \theta \leqslant 2\pi,
  R^{1}_i \leqslant r \leqslant R^2_i \},
  \label{eq:gencyl}
\end{equation}
where $r=r(s)$ and $\theta=\theta(s)$ are the cylindrical coordinates of the local coordinate system defined by vectors $\NN_i(s)$ and $\BB_i(s)$, and $R^1_i = R^1_i(s,\theta,t)$ and $R^2_i = R^2_i(s, \theta,t)$ denotes the inner and outer "radius", respectively. We emphasize that these radii are allowed to vary along the centerline (with $s$) and angularly (with $\theta$) and thus should be interpreted as generalized radii and the resulting domain as a generalized annular cylinder. The cross-section $C_i = C_i(t)$ of the channel (varying along with $s$) is now given by:
\begin{equation}
  C_i = \{ \boldsymbol{\lambda}_i  + r \cos(\theta) \NN_i + r\sin(\theta) \BB_i,
  0 < \theta \leqslant 2 \pi, \, R_i^{1} \leqslant r \leqslant R_i^{2}  \}.
  \label{eq:cross-section}
\end{equation}
We let $A_i = \vert C_i \vert$ denote the cross-section area. The inner and outer lateral boundaries of $\Omega_i$ are labeled $\Gamma^1_i$ and $\Gamma^2_i$, respectively, and we set $\Gamma_i = \Gamma^1_i \cup \Gamma^2_i$.

Finally, we construct the full perivascular space as the union of the separate channels $\Omega = \cup_{i=1}^n \Omega_i$, with $\Omega_i$ defined by \eqref{eq:gencyl}. We designate $\Gamma = \cup_{i=1}^n \Gamma_i$ to be its lateral boundary. Moreover, we define the perivascular cross-section area $A$, cross-section $C$, inner and outer radii $R^1$ and $R^2$, edge-wise by  
\begin{equation*}
  A |_{\Lambda_i} = A_i, \quad
  C |_{\Lambda_i} = C_i, \quad
  R^1 |_{\Lambda_i} =R_i^1, \quad
  R^2 |_{\Lambda_i} =R_i^2.
\end{equation*}

\subsection{Stokes--Brinkman perivascular flow equations (3D)}
\label{sec:model-full}

Consider the flow of an incompressible, viscous fluid in a saturated porous domain $\Omega \subset \R^3$ representing the PVS with porosity $\varphi \in (0, 1]$. The porosity describes the pore space accessible to the fluid; with $\varphi = 1$ corresponding to a non-porous/open/unrestricted domain. Let $\vvref$ denote the fluid velocity and $\pref$ a scaled fluid pressure (i.e.~the pressure divided by the fluid density) solving the following Stokes--Brinkman system \cite{whitaker1999theory} of time-dependent partial differential equations (PDEs) over $\Omega$:
\begin{subequations}
  \begin{align}
    \label{eq:stokes:a}
    \partial_t \vvref - \frac{\nu}{\varphi} \Delta \vvref+ \frac{\nu}{\kappa} \vvref + \nabla \pref &= 0, \\
    \nabla \cdot \vvref &= 0 . 
  \end{align}
  \label{eq:stokes}%
\end{subequations}%
In~\eqref{eq:stokes}, $\nu$ is the kinematic fluid viscosity, and $\kappa$ the permeability of the domain (typically depending on $\phi$). For non-porous/open domains, we have $\varphi=1$ and $\kappa \rightarrow \infty$, in which case \eqref{eq:stokes:a} simplifies to the momentum equation of the time-dependent Stokes equations:
\begin{equation*}
  \partial_t \vvref - \nu \Delta \vvref +  \nabla \pref = 0 \quad \text{ in } \Omega.
\end{equation*} 

We augment~\eqref{eq:stokes} with mixed boundary conditions as
follows. First, we introduce the stress
$\boldsymbol{\sigma}^{\text{ref}}_{\nn}$ defined relative to any
interface, with $\nn$ as the (outward pointing) unit normal vector, by:
\begin{align}
  \boldsymbol{\sigma}^{\text{ref}}_{\nn}(\vv, \p) = (\frac{\nu}{\varphi} \nabla \vv - \p I) \cdot \nn .
\end{align}
As boundary conditions at the domain inlets and outlets, we prescribe a given traction:
\begin{equation}
  \boldsymbol{\sigma}^{\text{ref}}_{\nn}(\vvref,\pref) = \tilde{p}^{\text{ref}} \nn \qquad \text{ at } \partial \Omega \setminus \Gamma,
  \label{eq:boundary-condition}
\end{equation}
which allows us, e.g., to define a given pressure drop over the length of the domain. For the sake of simplicity, we assume $\tilde{p}^{\text{ref}}$ to be constant at each inlet and outlet. Next, at the inner and outer lateral boundaries, we prescribe a given fluid velocity. Let $\ww$ denote the normal wall speed defined by the rate of change of inner and outer radius in the normal direction:
\begin{align}
\ww =  \begin{cases}
  \partial_t {R}^1 \nn \text{ on } \Gamma^{1}, \quad &\text{(inner wall movement)}\\
  \partial_t {R}^2 \nn \text{ on } \Gamma^{2}. \quad &\text{(outer wall movement)}
\end{cases}
\label{eq:wallspeed}
\end{align}
In particular, we assume the inner and outer radii are known at each time point. With this in hand, we then set the fluid velocity $\vvref$ to match the wall velocity $\ww$ on the lateral boundaries:
\begin{align}
\vvref = \ww \text{ on } \Gamma.
\end{align}
In the simulation scenarios of \Cref{sec:pvs-results}, we consider pulsating inner wall displacements -- with $R^1$ varying in time ($R^1=R^1(s,\theta,t)$) while $R^2$ is fixed in time ($R^2=R^2(s,\theta,t=0)$). The wall movement $\ww$ is then given by experimental data; or it may be calculated using a blood flow model that accounts for arterial wall displacement~\cite{formaggia2003one}.

\subsection{Stokes--Brinkman perivascular network equations}
\label{sec:model-network}

In this section, we introduce a network model for pulsatile perivascular flow; that is,~a geometrically-reduced model approximation to the Stokes--Brinkman flow equations~\eqref{eq:stokes} tailored to perivascular spaces. The detailed model derivation is available in Appendix \ref{sec:derivation}. The main ideas are as follows.

First, we make the following assumptions on each centerline $\Lambda_i$:
\begin{subequations}
    \begin{align}
    p^{\text{ref}}(r, \theta; s, t) &= p^{\text{ref}}(s, t) && s \in \Lambda_i, t > 0, (r, \theta) \in C_i(s, t) \label{eq:assumption-p},   \\
    v_s^{\text{ref}}(r, \theta; s, t) &=\hat{v}_s^{\text{ref}}(s,t) v^{vp}(r,\theta) && s \in \Lambda_i, t > 0, (r, \theta) \in C_i(s, t), \label{eq:assumption-v} 
    \end{align} 
\end{subequations}
where $v_s^{\text{ref}}$ denotes the axial component of $\vv^{\text{ref}}$.
The first assumption states that the pressure is constant along each cross-section. The second states that the velocity admits a certain separation of variables, where $v^{vp}$ is the velocity profile associated with a unit pressure drop in a pipe with cross-section $C(s)$, and $\hat{v}_s^{\text{ref}}$ gives a time-dependent scaling of this profile in the axial direction.

With these assumptions in hand, the full model equations can be integrated over the cross-section, and the derivatives moved out of the integral, yielding a one-dimensional model posed via a \textit{cross-section flux} $\q = \q(s, t)$ and the \textit{cross-section pressure} $\p = \p(s, t)$, defined as follows:
\begin{gather}
\begin{aligned}
\q=\int_{C} v_s^{\text{ref}} r \dr \dtheta,
\quad \p=\int_{C} \pref \, r \dr \dtheta.
\end{aligned}
\end{gather}
Here, $q$ and $p$ are defined by their restriction to each centerline $\Lambda_i$ and cross-section $C_i$; that is, $q=q_i$ and $p=p_i$ on $\Lambda_i$. The resulting model is a time-dependent Stokes--Brinkman equation solving for the cross-section pressure $\p$ and cross-section flux $q$. For each centerline $\Lambda_i$, the model reads: 
\begin{subequations}
  \begin{align}
    \partial_t \q + \mathcal{R}\q - \nu_{\text{eff}} \partial_{ss} \q +  \partial_s \p  &=0 && \text{ on } \Lambda_i,\label{eq:1d-pvs-1}\\
    \partial_s \q &= \partial_t A && \text{ on } \Lambda_i,\label{eq:1d-pvs-2}
  \end{align}
  \label{eq:1d-pvs}
\end{subequations}
where $\nu_{\text{eff}}(s,t)=\nu/(A (s, t)  \varphi)$. Physically, the source term $\partial_t A$ accounts for displacement of the fluid due to wall motion. The resistance $\mathcal{R}$ is a lumped parameter varying axially and in time,
\begin{align}
 {\mathcal{R}}(s,t)  = \frac{\nu}{q^{vp}(s, t)}+ 2\frac{\nu}{\kappa}. \label{eq:resistance}
\end{align}
Here,
\begin{align}
    q^{vp} = \int_{C} v^{vp}(r,\theta) r\dr\dtheta ,
\end{align}
is the cross-section flux associated with $v^{vp}$ in \eqref{eq:assumption-v}; the next section will show how $v^{vp}$ can be computed for any cross-section $C$.

It remains to specify conservation or continuity conditions at the (peri)vascular junctions. At each internal vertex $\vv_j \in \mathcal{I}$, we assume the pressure $p$ to be continuous, and in addition impose conservation of mass in terms of the flux $q$,
\begin{align}
\jump{q}_j=0 \text{ at } \vv_j \in \mathcal{I}.
\end{align}
Here, $\jump{q}_j$ is the generalized jump of $q$ at vertex $v_j$,
\begin{align}
\jump{q}_j = \sum_{\Lambda_i \in E_\text{in}(\vv_j)} q_i(\vv_j)-\sum_{\Lambda_i \in E_\text{out}(\vv_j)} q_i(\vv_j). \label{eq:jump}
\end{align}
At each boundary vertex $\vv_j \in \partial \V$, we assign an axial traction condition corresponding to a cross-section average of the original boundary condition~\eqref{eq:boundary-condition}:
\begin{align}
    \nu_{\text{eff}}\partial_s q- p = \tilde{p}^{\text{ref}} \text{ at } \vv_j.
\end{align}

Comparing the three-dimensional reference model \eqref{eq:stokes} with the network equations~\eqref{eq:1d-pvs}, we see that the axial dissipation term $\Delta v_s^{\text{ref}}$ decompose into two parts: $\partial_{ss} \q$ and $\mathcal{R} \q$. The first term, $\partial_{ss} q$, accounts for viscous dissipation of energy due to changes in the axial flow speed. In our applications, this term is generally small. In fact, it is nonzero only in specific cases of pulsatile flow. The second term, $\mathcal{R} q$, accounts for (i) energy dissipation due to the no-slip boundary condition on the inner and outer walls and (ii) resistance due to the pore network. In our applications, this term is typically large. The contribution of the no-slip condition to the network resistance gives rise to the following remark.
\vspace{0.5em}
\begin{remark}[Both Stokes flow and Stokes--Brinkman flow yield Stokes--Brinkman type network models]
Consider the network resistance $\mathcal{R}$ defined by~\eqref{eq:resistance} as the sum of two contributions: the resistance inversely associated with the characteristic cross-section velocity profile $v^{vp}$ and the resistance due to the pore network. For open (non-porous) channels, $\kappa \rightarrow \infty$; thus the latter contribution vanishes. The first term remains, meaning that the network model corresponding to Stokes flow still has a non-negative resistance $\mathcal{R}$. This resistance stems from the no-slip boundary condition on each cross-section, and depends on the shape and size of these through $v^{vp}$.     
\end{remark}

\subsection{Impact of perivascular shape and size on the resistance}
\label{sec:model-resistance}

Consider the flow driven by a constant pressure drop through a domain with constant cross-section $C$. Inserting the separation of variables $v_s^{\text{ref}}=\hat{v}_s^{\text{ref}}(s,t) v^{vp}(r,\theta)$ into the Stokes--Brinkman equations \eqref{eq:stokes}, we find that the velocity profile $v^{vp}$ associated with a cross-section $C$ solves
\begin{align*}
  -\frac{1}{\varphi} \Delta v^{vp}+\frac{1}{\kappa} v^{vp} &= -1 \quad && \text{ on } C, \\
  v^{vp} &= 0 \quad && \text{ on } \partial C.
\end{align*}
After solving this either analytically or numerically, one can compute the velocity profile cross-section flux $q^{vp}$ and hence the resistance $\mathcal{R}$ \eqref{eq:resistance}.

The resistance thus depends both on the shape and size of $C$. Their influence can be separated as follows. Let $\tilde{C}$ denote the non-dimensionalized cross-section, i.e. $C$ scaled so that it has unit inner radius. Letting $\tilde{\mathcal{R}}$ denote the associated resistance, one then has ~\cite{tithof2019hydraulic}
\begin{align*}
\mathcal{R} = \tilde{\mathcal{R}}/ (R^1)^4,
\end{align*}
where the numerator $\tilde{\mathcal{R}}$ only depends on the shape of the domain $C$. In our computations, we typically assume the shape of $C$ is fixed in time, meaning that the time-dependency of $\mathcal{R}$ enters through the denominator $(R^1)^4$.

\label{sec:num-resist}
\begin{figure}[h]
\begin{subfigure}{0.45\textwidth}
  \centering
  \includegraphics[width=0.99\textwidth]{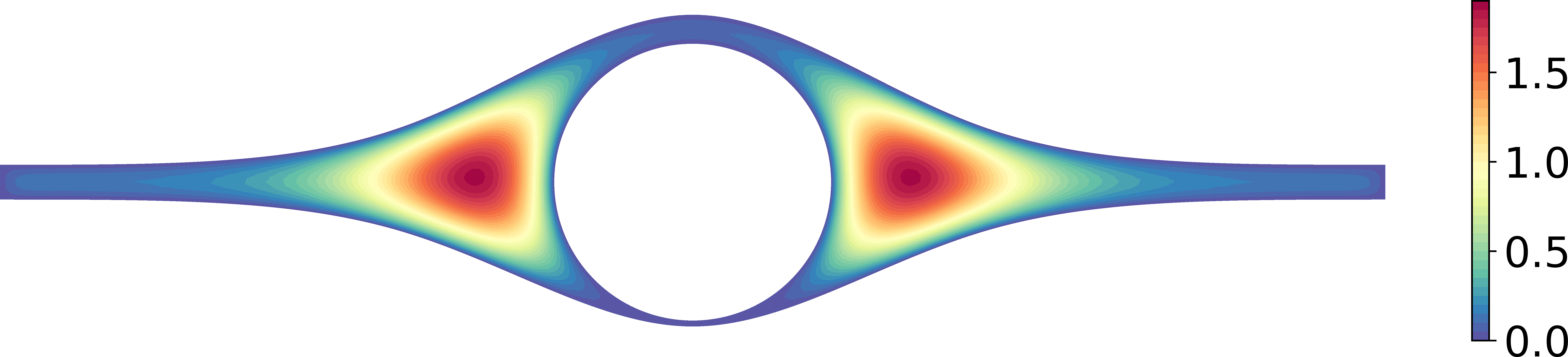}
  \caption{Idealized arterial PVS}
\end{subfigure}
\hfill
\begin{subfigure}{0.45\textwidth}
  \centering
  \includegraphics[width=0.99\textwidth]{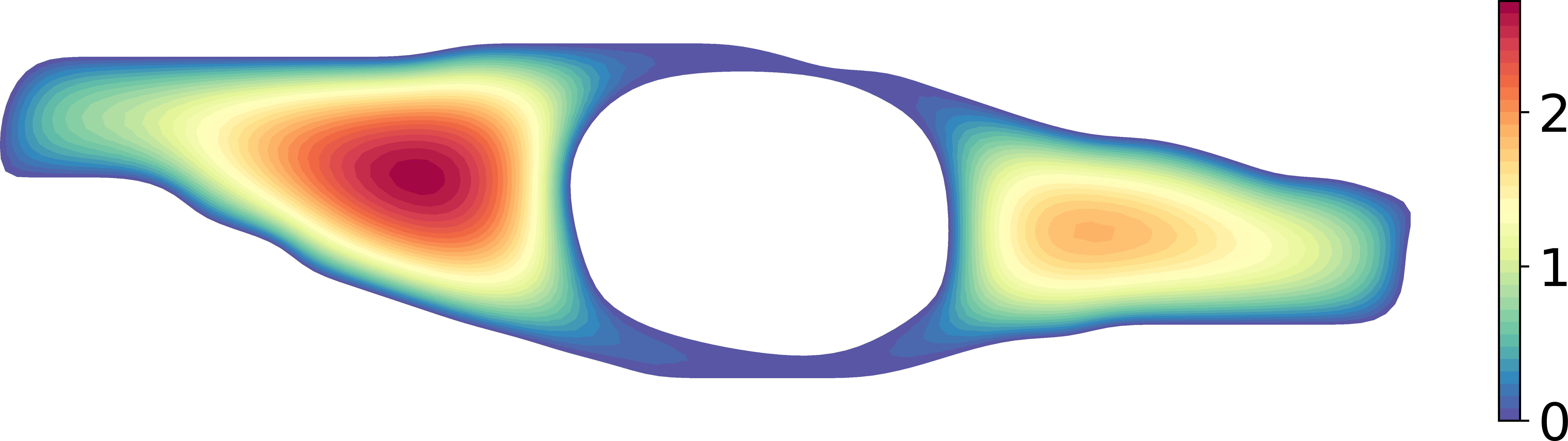}
  \caption{Image-based arterial PVS}
\end{subfigure}
\vspace{3em}
\begin{subfigure}{0.45\textwidth}
  \centering
  \includegraphics[width=0.99\textwidth]{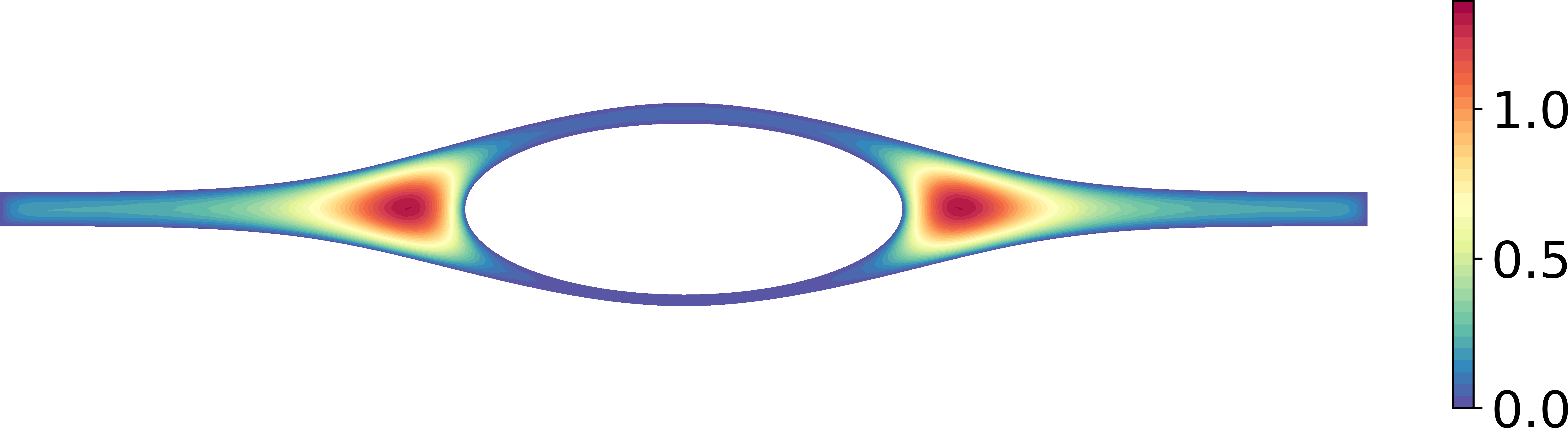}
  \caption{Idealized venous PVS}
\end{subfigure}
\hfill
\begin{subfigure}{0.45\textwidth}
  \centering
  \includegraphics[width=0.99\textwidth]{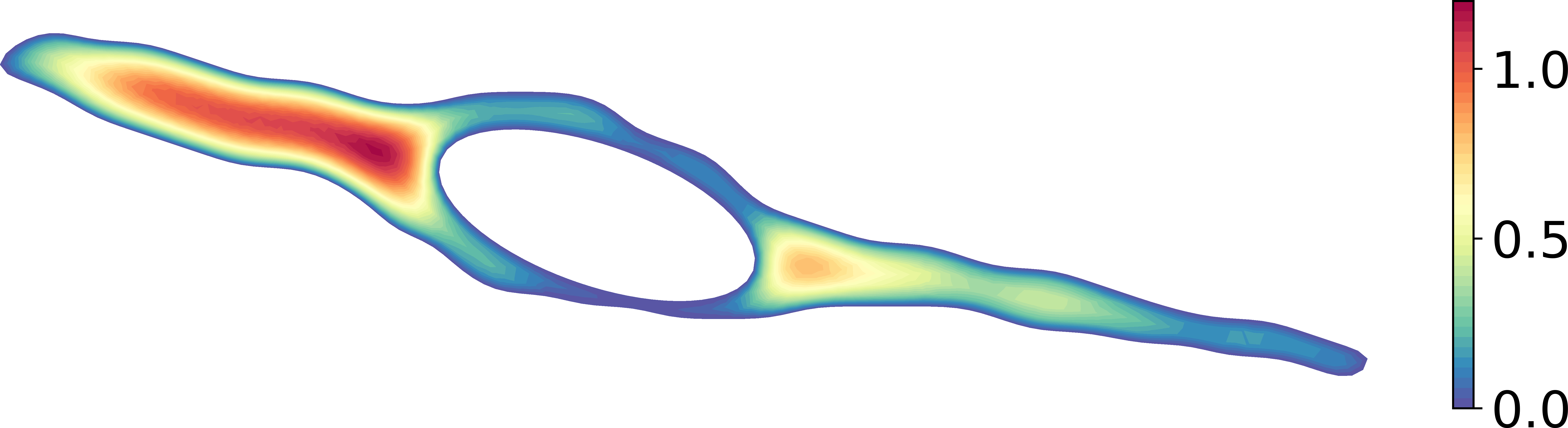}
  \caption{Image-based venous PVS}
\end{subfigure}
\caption{Velocity profile $v^{vp}$ associated with idealized and image-based pial PVS cross-section shapes. The top and bottom rows show shapes associated with arterial and venous PVSs, respectively. We see that the asymmetry of the image-based pial artery PVS yields an increase in the velocity profile magnitude. We therefore expect a considerably lower resistance offered by this domain.}
\label{fig:cross-sections}
\end{figure}
\begin{table}
\begin{subtable}{0.99\textwidth}
\centering
\hfill
\begin{tabular}{l | c c }
  Domain & $\mathcal{R}_{R^1=1\text{mm}^2}$  & $\mathcal{R}_{A=100\text{mm}^2}$  \\
  \midrule
  Idealized arterial PVS      &  3.7e-03 & 1.2e-05 \\
  Image-based arterial PVS    &  3.5e-04 & 7.3e-06 \\
  Idealized venous PVS        &  2.0e-03 & 1.9e-05 \\
  Image-based venous PVS      &  3.5e-04 & 1.8e-05 
\end{tabular}
\hfill
\end{subtable}
\caption{The resistance parameter $\mathcal{R}$ computed for the domains shown in Figure \ref{fig:cross-sections}. The resistance parameter was computed using a reference radius $R^1=1$mm (middle column) and a reference area $A=100$mm$^2$ (right column). We observe that the image-based pial artery PVS yields substantially lower resistance than its idealized counterpart, even when the cross-section areas for each domain are normalized.}
\label{tab:resistance}
\end{table}
In Figure \ref{fig:cross-sections}, we show the velocity profile $v^{vp}$ computed on idealized and image-based cross-sections of pial arteries and veins, using in-vivo human image data as in~\cite{vinje2021brain, bedussi2018paravascular}. In both cases, we assume that the domain is open, i.e., that $\varphi=1$. Table \ref{tab:resistance} shows the resistance parameters associated with each cross-section. The middle column gives the resistance values when the inner radius is scaled so that $R_1=1$mm.  Interestingly, the image-based periarterial resistance is up to an order of magnitude lower than the resistance computed using the idealized geometries. This observation can be attributed to the effects of cross-section asymmetry, as highlighted by \citet{tithof2019hydraulic}. However, resistance also decreases with cross-section area. To isolate the effect of asymmetry, we therefore show in the right-most column the resistance for cross-sections with their area normalized to 100 mm$^2$. We still observe a close to 50$\%$ lower resistance. 

\begin{figure}
\begin{subfigure}{0.45\textwidth}
\includegraphics[width=0.99\textwidth]{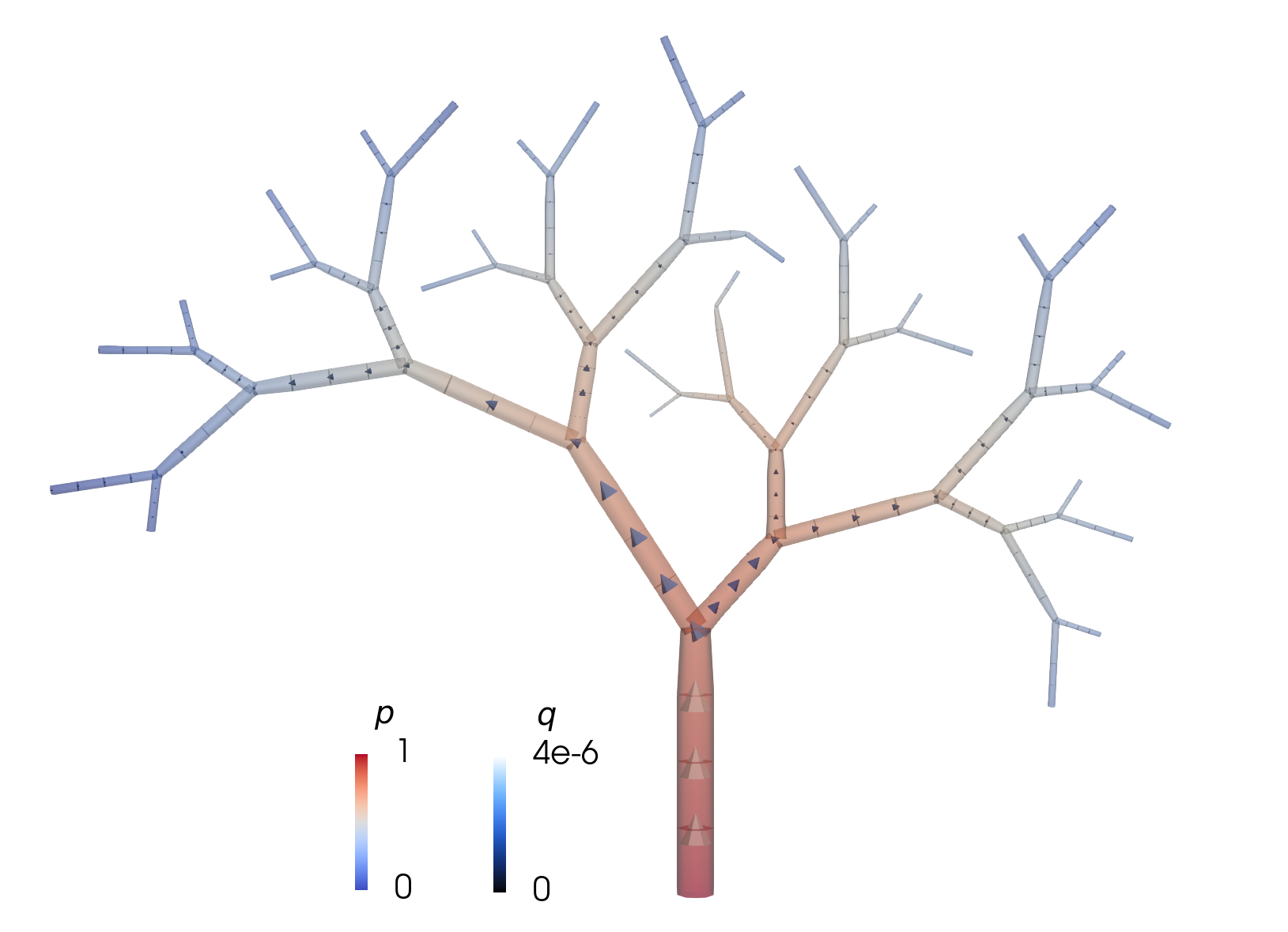}
\caption{Idealized cross-section}
\end{subfigure}
\begin{subfigure}{0.45\textwidth}
\includegraphics[width=0.99\textwidth]{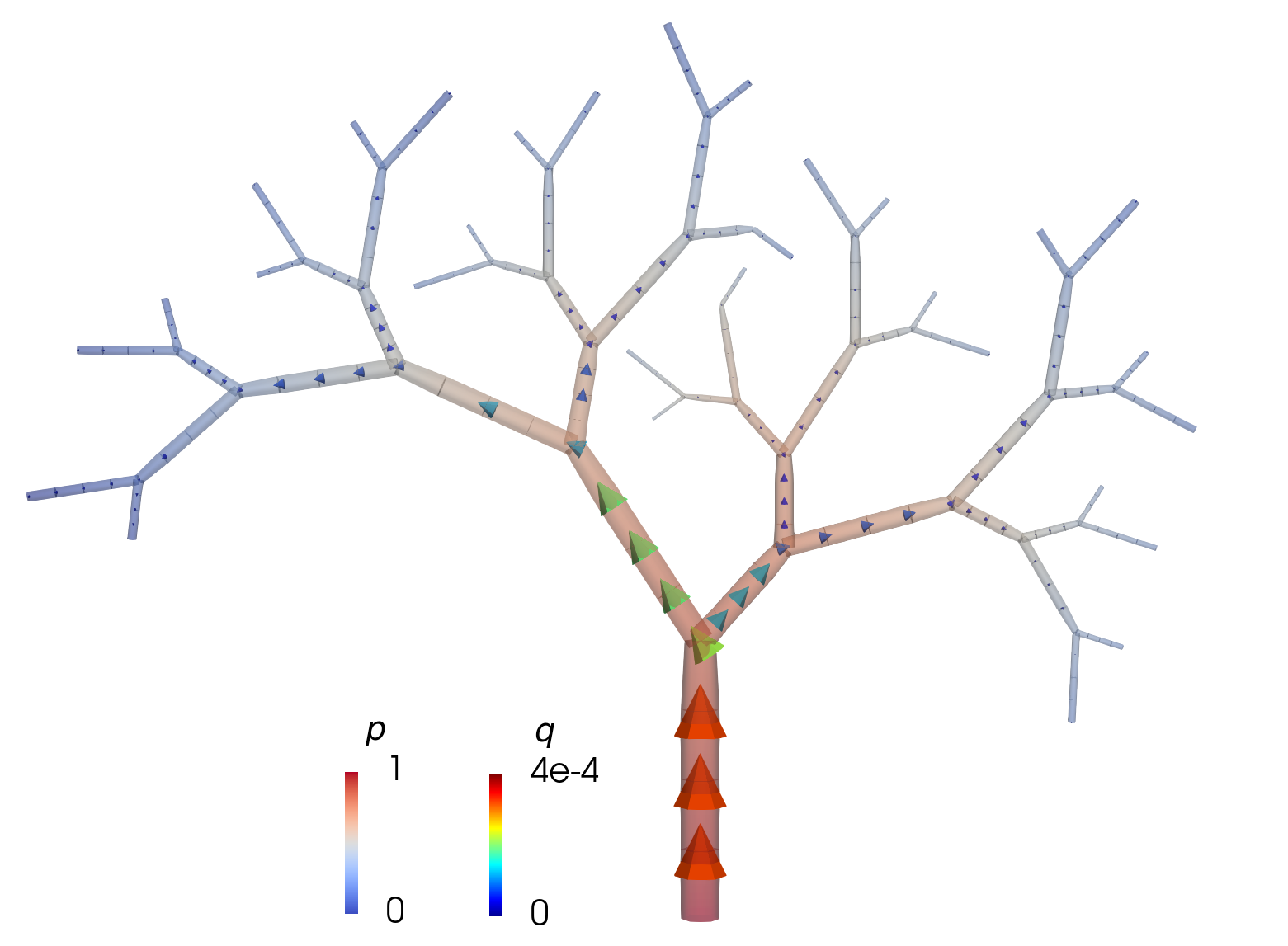}
\caption{Image-based cross-section}
\end{subfigure}
\caption{Pressure $p$ and cross-section flow $q$ due to a fixed pressure drop through an arterial tree, computed using resistance due to idealized (left) and image-based (right) arterial PVS cross-sections. Due to the lower resistance in the image-based cross-sections, the pressure drop in this tree yields a larger cross-section flux.}
\label{fig:alex-tree}
\end{figure}
Figure \ref{fig:alex-tree} shows simulation results using the image-based and idealized resistance parameters to compute the flow induced by a constant pressure gradient over an arterial tree\footnote{The arterial tree was generated using \url{https://gitlab.com/ValletAlexandra/NetworkGen/}}. We observe that the lower resistance associated with the image-based cross-section yields an increase in cross-section flow by two orders of magnitude. To obtain high accuracy results from 1-dimensional models, a careful choice of the resistance parameter is thus mandatory. 

\section{Uniform well-posedness and approximation of Stokes--Brinkman network models}
\label{sec:graph}

In this section, we focus on mathematical and numerical properties of the Stokes--Brinkman perivascular network equations~\eqref{eq:1d-pvs}. To facilitate the analysis, we introduce a graph calculus-based formulation of our network model, and we therefore first define some general concepts from graph calculus~\cite{friedman2004calculus} in \Cref{sec:graph-calc}, before presenting the abstract model formulation in~\Cref{sec:graph-calculus-model}. The well-posedness and stability of primal and dual formulations of this model are studied in \Cref{sec:graph-analysis-primal} and \ref{sec:graph-analysis-dual}. Importantly, we show that the formulations are uniformly stable with respect to the network topology in terms of the number of bifurcations. In \Cref{sec:graph-conv}, we compare and evaluate numerical properties of the primal and dual discretizations. Both methods converge with respect to the meshsize $h$. The discretizations were implemented in FEniCS \cite{AlnaesBlechta2015a}, using graphnics \cite{gjerde-joss} to construct the jump conditions and fenics$\_$ii \cite{kuchta2021assembly} to assemble the resulting block matrices.



\subsection{Graph calculus and graph finite elements}
\label{sec:graph-calc}

Consider an oriented spatial graph $\mathcal{G} = (\V, \E)$ with
vertices $\V =\{\vv_1, ..., \vv_m \}$ for $\vv_j \in \R^3$ ($j = 1,
\dots, m$) and edges $\E = \{\Lambda_1, \Lambda_2, ..., \Lambda_n \}$
parametrized by $s$. Let $C^{k}(\E)$ denote the space of functions
that are $k$-times continuous on each curve $\Lambda_i$. Further, let
$L^2(\V)$ denote the set of functions that are finite on each $\vv_j
\in \V$.

\subsubsection{Graph gradient and divergence}
\label{sec:graph-derivatives}

We define the \textit{graph gradient} $\GradG : C^{k}(\E) \rightarrow C^{k-1}(\E)$ by
\begin{equation*}
  \GradG p = \partial_s p \text{ on } \mathcal{E},
\end{equation*}
and a \textit{graph divergence} $\DivG: C^k(\E) \rightarrow C^{k-1}(\E) \times L^2(\V)$ by
\begin{equation*}
  \DivG q =
  \begin{cases}
    \partial_s q \text{ on } \mathcal{E}, \\
    \jump{q}_j \text{ on } \vv_j \in \mathcal{V},
  \end{cases}
\end{equation*}
where $\jump{q}_j$ is the jump of $q$ defined in \eqref{eq:jump}. We also define the \textit{edge Laplacian} $\Delta_{\mathcal{E}} :  C^{k}(\E) \rightarrow C^{k-2}(\E)$ by
\begin{equation*}
  \Delta_{\mathcal{E}} p = \partial_{ss} p \text{ on } \mathcal{E}.
\end{equation*}
Formally, the gradient and edge Laplacian map functions from $\mathcal{E}$ to $\mathcal{E}$. The divergence maps functions from $\mathcal{E}$ to $\mathcal{G}$, where $\mathcal{G}$ consists of vertices and edges. These operators reflect the mixed-dimensional structure of the network (consisting of one-dimensional edges connected by zero-dimensional vertices), and can be seen as a special case of the operators introduced by \citet{boon2021functional}. 

\subsubsection{Sobolev spaces on graphs}

We can use these differential operators to define inner products and Sobolev spaces on the graph. Recall that $\Lambda$ denotes the extended graph, i.e.,
\begin{equation*}
\Lambda = \V \cup \E = \cup_{i=1}^n \, \text{closure}( \Lambda_i).
\end{equation*}
%
Given a measurable function $u$ defined over $\Lambda$, let $u_i$ denote the restriction of $u$ to $\Lambda_i$. We define the \textit{standard} inner product
\begin{equation*}
    (u,v)_\Lambda = \sum_{i=1}^n (u_i, v_i)_{\Lambda_i} =\sum_{i=1}^n \int_{\Lambda_i} u_i v_i \ds,
\end{equation*}
which gives rise to the \textit{standard}  $L^2$-space
\begin{equation*}
  L^2(\Lambda)=L^2(\E) = \{ u \text{ meas.}: \int_\Lambda u^2\ds < \infty \}.
\end{equation*}
We note that we can identify  $L^2(\Lambda)$ with $L^2(\E)$ as they belong to the same equivalence class. 

Introducing a graph measure allows us to take into account the fact that edges and vertices have different dimensions. The graph measure $\dG$ \cite{friedman2004calculus} is given by
\begin{equation*}
    \int_{\mathcal{G}} u \dG = \int_{\mathcal{E}} u \dE + \int_{\mathcal{V}} u \dV, 
\end{equation*}
with edge and vertex measures
\begin{equation*}
    \int_{\mathcal{E}} u \dE = \sum_{i=1}^n \int_{\Lambda_i} u_i \ds, \quad \int_{\mathcal{V}} u \dV = \sum_{j=1}^m u(\vv_j).
\end{equation*}
The graph measure thus naturally induces a graph inner product
\begin{equation*}
    (u,v)_\mathcal{G} = \sum_{i=1}^n (u,v)_{\Lambda_i} + \sum_{j=1}^m u(\vv_j) v(\vv_j),
\end{equation*}
and the corresponding $L^2$ space
\begin{equation*}
    L^2(\mathcal{G}) = \{ u \text{ meas. }: \sum_{i=1}^n \norm{u}{L^2(\Lambda_i)} + \sum_{j=1}^m \vert u(\vv_j)\vert^2  <\infty\}.
\end{equation*}
We will also use the notation $u = (u_\E, u_\V) \in L^2(\mathcal{G})$ to separate the edge and vertex components of $u$. 

We now construct different types of Sobolev spaces on $\mathcal{G}$. We use $H^1(\E)$ and $H^2(\E)$ to denote the broken Sobolev spaces 
\begin{align*}
    H^1(\E) &= \{u \in L^2(\E): \partial_{s} u \in L^2(\E) \}, \\
    H^2(\E) &= \{u \in L^2(\E): \partial_{s} u \in L^2(\E), \partial_{ss} u \in L^2(\E) \},
\end{align*}
and $H^1(\Lambda)$ is defined as:
\begin{equation*}
    H^1(\Lambda) = \{ u \in L^2(\Lambda) : \GradG u \in L^2(\Lambda)\}.
\end{equation*}
The latter space is known from e.g.~\citep{arioli2018finite}, and has the norm
\begin{equation*}
    \norm{u}{H^1(\Lambda)}^2 = \norm{u}{L^2(\Lambda)}^2 + \norm{\GradG u}{L^2(\Lambda)}^2.
\end{equation*}
We use the notation $H_0^1(\Lambda)$ to denote $H^1$-functions with zero trace on $\partial V$. While $L^2(\Lambda)$ is equivalent to $L^2(\E)$, we note that $H^1(\E)$ and $H^1(\Lambda)$ are not equivalent. Indeed, recalling from standard Sobolev theory that $H^1(\Lambda)\subset C^0(\Lambda)$, we find $H^1(\Lambda) \subset H^1(\E)$, as $H^1(\E)$ functions can be discontinuous across vertices.

Next, let $H(\mathrm{div};\mathcal{G})$ denote the space
\begin{equation*}
    H(\mathrm{div}; \mathcal{G}) = \{ q \in L^2(\mathcal{\mathcal{E}}) : \DivG q \in L^2(\mathcal{G})\},
\end{equation*}
with the norm
\begin{gather}\label{eq:simple_hdiv_norm}
  \begin{aligned}
    \norm{q}{H(\mathrm{div};\mathcal{G})}^2 &= \norm{q}{L^2(\mathcal{E})}^2 + \norm{\DivG q}{L^2(\mathcal{G})}^2 \\
    & = \sum_{i=1}^n \norm{q}{L^2(\Lambda_i)}^2 + \sum_{i=1}^n \norm{\partial_s q}{L^2(\Lambda_i)}^2+ \sum_{j=1}^m \vert  \jump{q}_j \vert^2.
  \end{aligned}
\end{gather}

\begin{figure}
\begin{overpic}[scale=0.5,unit=1mm]{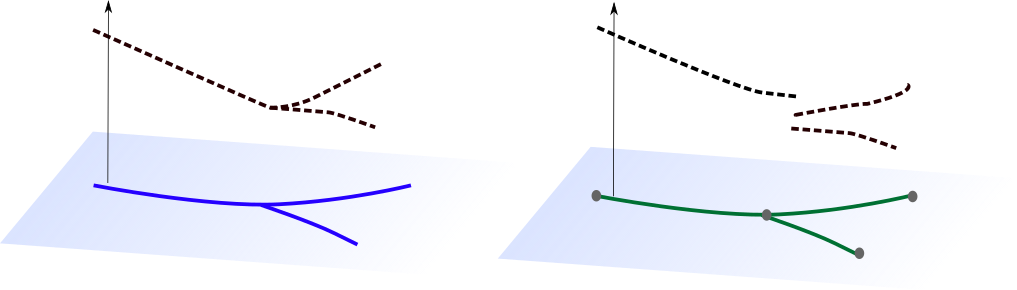}
\put(20, 5){\textcolor{blue}{$\Lambda$}}
\put(65, 5){\textcolor{green}{$\Lambda_1$}}
\put(78, 3){\textcolor{green}{$\Lambda_2$}}
\put(80, 10){\textcolor{green}{$\Lambda_3$}}
\put(75, 9){\textcolor{gray}{$x_2$}}
\put(55, 10){\textcolor{gray}{$x_1$}}
\end{overpic}
\caption{Examples of functions that are in $H^1(\Lambda)$ (left) and $H(\mathrm{div}; \mathcal{G})$ (right).}
\label{fig:sobolev}
\end{figure}
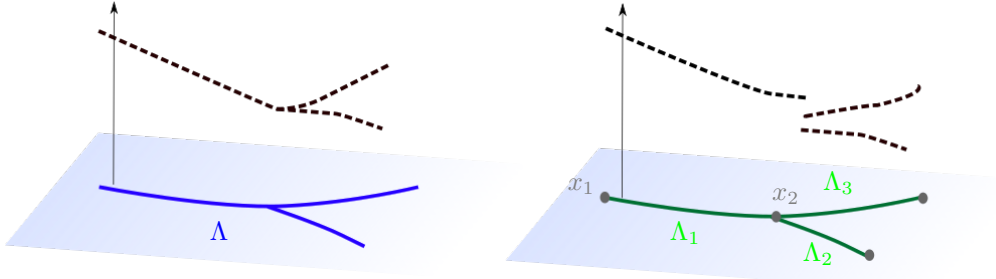
Figure \ref{fig:sobolev} shows examples of functions in $H^1(\Lambda)$ versus $H(\mathrm{div};\mathcal{G})$. The main difference between these spaces is that $u \in H(\mathrm{div}; \mathcal{G})$ can be discontinuous at the vertices.  We note that $H^1(\E)$ and $H(\mathrm{div};\mathcal{G})$ are equivalent, as $H^1(\E)$-functions have bounded values at $\partial \Lambda_i$ (and hence bounded jumps across $\mathcal{V}$). However, we keep the $H(\mathrm{div};\mathcal{G})$-notation to emphasize the connection to standard methods for dual mixed formulations. Moreover, we will see that an appropriately weighted $H(\mathrm{div};\mathcal{G})$-norm is required for uniform stability. 

Having $H^1(\Lambda)$ and $H(\mathrm{div};\mathcal{G})$  defined, it is easy to see that the following integration by parts formula holds.
\begin{lem}[Integration by parts] \label{lemma:integration-by-parts}
For $p\in H^1_0(\Lambda)$ and $q\in H(\mathrm{div};\mathcal{G})$ there holds that
\begin{equation*}
  \int_\mathcal{G} (\DivG q) \,  p \dG = -\int_{\mathcal{E}} q \, (\GradG p) \dE.
\end{equation*}
\end{lem}
\begin{proof}
A direct calculation shows that
\begin{gather*}
\begin{aligned}
  - \int_{\mathcal{G}} q \, (\GradG p) \dE
  &= - \sum_{i=1}^n \int_{\Lambda_i} q_i \, \partial_s p_i \ds
  =  \sum_{i=1}^n \int_{\Lambda_i} \partial_s q_i \, p_i \ds + \sum_{j=1}^m \jump{pq}_j \\
  &=  \sum_{i=1}^n \int_{\Lambda_i} \partial_s q\, p \ds + \sum_{j=1}^m p(v_j) \jump{q}_j
  =  \int_\mathcal{G} (\DivG q) p \dG ,
\end{aligned}
\end{gather*}
where we used that $p$ is continuous over the graph.
\end{proof}

\subsubsection{Finite element spaces on graphs}

We now introduce finite element meshes and finite element spaces defined relative to the graph. Let $\Lambda^h$ be a finite element mesh of the centerline $\Lambda$, composed of mesh segments $\Lambda_i^h$, one for each centerline $\Lambda_i$. Each mesh segment $\Lambda_i^h$ is a mesh consisting of intervals embedded in $\R^3$. Relative to $\Lambda_i^h$, we define $CG_{k}(\Lambda_i^h)$ to be the space of continuous piecewise polynomials of degree $k$ defined relative to $\Lambda_i^h$, i.e.
\begin{align*}
CG_{k}(\Lambda_i^h) = \lbrace v^h \in C^0(\Lambda_i), \, v^h \vert_T \in P_k(T)  \text{ for } T \in \Lambda_i^h \rbrace,
\end{align*}
and similarly $CG_{k}(\Lambda^h)$ to be the space of continuous piecewise polynomials of degree $k$ defined relative to $\Lambda^h$, i.e.
\begin{align*}
CG_{k}(\Lambda^h) = \lbrace v^h \in C^0(\Lambda), \, v^h \vert_T \in P_k(T)  \text{ for } T \in \Lambda^h \rbrace .
\end{align*}
We define $DG_{k}(\Lambda_i^h)$ to be the space of discontinuous piecewise polynomials of degree $k$ on $\Lambda_i^h$, i.e.
\begin{align*}
DG_{k}(\Lambda_i^h) = \lbrace v^h \in L^2(\Lambda_i), \, v^h \vert_T \in P_k(T)  \text{ for } T \in \Lambda_i^h \rbrace,
\end{align*}
and $DG_{k}(\Lambda^h) = \cup_{i=1}^n DG_{k}(\Lambda_i^h)$ to be the equivalent space on $\Lambda^h$.

\subsection{A graph calculus formulation of the Stokes--Brinkman network model}
\label{sec:graph-calculus-model}

With the graph calculus notation introduced in~\Cref{sec:graph-derivatives}, the time--dependent Stokes--Brinkman model \eqref{eq:1d-pvs} can be succinctly expressed as: for $t > 0$, find $(\q, \p)$ defined over $\mathcal{G}$ such that
\begin{subequations}
  \begin{align}
    \partial_t \q + \mathcal{R} \q - \nu_{\text{eff}} \Delta_\mathcal{E} \q + \GradG \p   &= 0 && \text{ on } \mathcal{E}, \label{eq:1d-stokes-1}\\
    \DivG \q &= f && \text{ on } \mathcal{G} \label{eq:1d-stokes-2},
  \end{align}
  \label{eq:1d-stokes}
\end{subequations}
where
\begin{align}
  f = \begin{cases}
    \partial_t A \text{ on } \E, \\
    0 \text{ on } \V.
  \end{cases}
\end{align}
Moreover, this system can be reduced to a time-dependent hydraulic network model. Since $\Delta_\E q = \GradG (\DivG q)$ on the edges $\E$, \eqref{eq:1d-stokes-2} gives that $\Delta_\mathcal{E} q= \nabla_\mathcal{G} f$. Thus, $(\q, \p)$ solving~\eqref{eq:1d-stokes} also solve:
\begin{subequations}
  \begin{align}
    \partial_t \q + \mathcal{R} \q +  \nabla_{\mathcal{G}} p   &= g && \text{ on } \mathcal{E}, \label{eq:1d-hyd-1}\\
    \DivG \q &= f && \text{ on } \mathcal{G}, \label{eq:1d-hyd-2}
  \end{align}
  \label{eq:1d-hyd}%
\end{subequations}%
where by definition $g = \nu_{\text{eff}} \GradG f = \nu_{\text{eff}} \partial_s \partial_t A$.

\vspace{0.5em}
\begin{remark}[Relation to quantum graphs]
  \label{rmrk:quantum_graph}
  The system \eqref{eq:1d-hyd} can be interpreted as a quantum graph with the differential operator $(q,p)\mapsto (\partial_t q + \mathcal{R}q + \partial_s p, \partial_s q)$. The bifurcation condition is equivalent to the standard Neumann-Kirchhoff conditions. In the stationary case, i.e. $\partial_t q=0$, the flux can be eliminated, yielding the simpler system
  \begin{equation}
    - \partial_s( \mathcal{R}^{-1} \partial_{s} \p) = \tilde{f} \qquad \text{ on } \Lambda_i,
  \end{equation}
  with $\tilde{f}= f - \partial_s (\nu_{\text{eff}} \mathcal{R}^{-1} \partial_s f).$ This corresponds to a quantum graph with the Laplacian $p\mapsto - \partial_s( \mathcal{R}^{-1} \partial_{s} \p)$ as the differential operator \cite{berkolaiko2013introduction}. The analysis we provide herein can be viewed as an extension of previous work on quantum graphs \cite{arioli2018finite} to the case where the differential operator is the primal and dual mixed Laplacian. 
\end{remark}


In the following sections, we will study the well-posedness and stability of (discretizations of) the hydraulic network model~\eqref{eq:1d-hyd} and in part~\eqref{eq:1d-stokes}. To simplify the exposition, we will only consider the stationary case ($\partial_t q=0$) with homogeneous Dirichlet boundary conditions ($p=0$ on $\partial \V$). We will use the saddle point theory from ~\citep{boffi-brezzi-fortin}, expressing the models in the general abstract mixed form: find $q \in V$, $p \in M$ such that
\begin{gather}
  \begin{aligned}
    a(\q, \psi) + b (\psi, p) &= L(\psi),  \\
    b\left(\q, \phi \right) &= F(\phi),
  \end{aligned}
  \label{eq:abstract}
\end{gather}
for all $\psi \in V$, $\phi \in M$. Here $V$ and $M$ are Hilbert spaces with inner products $(\cdot, \cdot)_V$ and $(\cdot, \cdot)_M$, respectively, $a : V \times V \rightarrow \R$ and $b : V \times M \rightarrow R$ are bilinear forms, and $L \in V^{\ast}$ and $F \in M^{\ast}$ are given functionals. We can and will study the hydraulic network formulation~\eqref{eq:1d-hyd} in both primal and dual variational form, while the Stokes--Brinkman model \eqref{eq:1d-stokes} requires the dual form. 

Given the discrete function spaces $V^h$ and $M^h$, with inner products $(\cdot, \cdot)_{V^h}$ and $(\cdot, \cdot)_{M^h}$, and norms $\| \cdot \|_{V^h}$ and $\| \cdot \|_{M^h}$, we will also consider discretizations of~\eqref{eq:abstract} i.e. the problem of finding discrete solutions $(\q^h, \p^h) \in V^h \times M^h$ such that~\eqref{eq:abstract} holds for all $\psi \in V^h$ and $\phi \in M^h$. The discrete system is then associated with a discrete inf-sup constant $\beta^h$, defined by  
\begin{align}
\beta^h = \inf_{0 \neq (\q^h, \p^h) \in W^h}\sup_{0 \neq (\psi, \phi) \in W^h} \frac{\vert a(\q^h, \psi) + b(\q^h, \phi)+b(\psi, \p^h) \vert}{(\norm{\q^h}{V^h}+\norm{\p^h}{M^h})(\norm{\psi}{V^h}+\norm{\phi}{M^h})},
\end{align}
where $W^h = V^h \times M^h$. The discretization is said to be inf-sup stable if there exists some $\beta> 0$ such that $\beta^h\geq \beta$ for any $h>0$. The inf-sup constant can be equivalently expressed as $\beta^h = \vert \xi^h_{\text{min}} \vert$ where $\xi^h_{\text{min}}$ is the smallest in modulus eigenvalue of the following generalized eigenvalue problem: find $(\q^h, \p^h) \in W^h$, $\xi^h\in\mathbb{R}$ such that
\begin{align}
a(\q^h, \psi) + b(\q^h, \phi)+b(\psi, \p^h) = \xi^h \left( (\q^h, \psi)_{V^h} + (\p^h, \psi)_{M^h} \right) \label{eq:eigenvalue}
\end{align}
for all $(\psi, \phi) \in W^h$. We will use this eigenvalue problem to provide numerical evidence that finite element discretizations are uniformly stable with respect to both $h$ and the network topology.

\subsection{Well-posedness of primal formulations of the hydraulic network models}
\label{sec:graph-analysis-primal}
In this section, we focus on the primal formulation and its stability properties.
\subsubsection{A primal formulation of the hydraulic network model}

We begin by presenting a primal mixed formulation of the stationary ($\partial_t q = 0$) hydraulic network model~\eqref{eq:1d-hyd} with homogeneous Dirichlet conditions ($p=0$ on $\partial \V$) based on the function space pairing $L^2(\Lambda) \times H^1_0(\Lambda)$. Multiplying~\eqref{eq:1d-hyd} by test functions $\psi \in L^2(\Lambda)$ and $\phi \in H^1_0(\Lambda)$, integrating over the graph, and using the integration by parts (Lemma \ref{lemma:integration-by-parts}), give the \textit{primal mixed variational formulation}: find $q \in L^2(\Lambda)$, $p \in H_0^1(\Lambda)$ such that
\begin{gather}
  \begin{aligned}
    (\mathcal{R} q,\psi)_\mathcal{E} +  (\nabla_\mathcal{G} p, \psi)_\mathcal{E}   &= (g,\psi)_\E, \\
    (q, \nabla_\mathcal{G} \phi)_\mathcal{E} &= (- f,\phi)_\mathcal{G} .
  \end{aligned}
  \label{eq:varform-primal}
\end{gather}
Note that for any $u,v \in L^2(\E)$, 
\begin{equation}
  (u,v)_\E = \sum_{i=1}^n (u,v)_{L^2(\Lambda_i)} = (u, v)_{L^2(\Lambda)}
 \end{equation}
We thus observe that~\eqref{eq:varform-primal} fits the general abstract framework~\eqref{eq:abstract} when identifying $V = L^2(\Lambda)$ and $M = H^1_0(\Lambda)$, and
\begin{equation*}
  a(q, \psi) = (\mathcal{R} q,\psi)_\Lambda, \quad 
  b(\psi, p) = (\psi, \partial_s p)_\Lambda, \quad
  L(\psi)  = (g,\psi)_\Lambda, \quad
  F(\phi) = (-f, \phi)_\Lambda. \quad
\end{equation*}

Our first theoretical result shows that the system~\eqref{eq:varform-primal} is well-posed, with uniform stability and inf-sup constants in resistance-weighted norms. 
\begin{theorem}
  \label{thm:primal_inf_sup}
  Let $V=L^2(\Lambda)$ and $M=H^1_0(\Lambda)$ be endowed with the weighted norms
  \begin{gather}\label{eq:primal_inf_sup_norms}
    \begin{aligned}
      \norm{p}{M} &= \norm{\mathcal{R}^{-1/2} \nabla_\mathcal{G} p}{L^2(\Lambda)}^2, \\
      \norm{q}{V} &= \norm{\mathcal{R}^{1/2} q}{L^2(\Lambda)}.
    \end{aligned}
  \end{gather}
  Given $f \in L^2(\mathcal{G})$ and $g \in L^2(\Lambda)$, there then exists a unique solution $q \in V$ and $p \in M$ to the primal mixed variational formulation \eqref{eq:varform-primal}. Moreover, the coercivity and inf-sup constants  are uniform with respect to the size and topology of the network.
  \label{thm:primal}
\end{theorem}
\begin{proof}
    The proof is by verifying the Brezzi conditions. First, note that by definition
    \begin{align}
        a(q,q) = \norm{\mathcal{R}^{1/2}q}{L^2(\Lambda)}^2 = \norm{ q }{V}^2,
    \end{align}
    which yields independent of $\mathcal{G}$ the coercivity constant one. Due to the boundary conditions the Poincar\'e inequality guarantees that $\norm{\mathcal{R}^{-1/2} \GradG \phi}{L^2(\Lambda)}$ is a norm on $M$. Then, for any $\phi \in M$,
    letting $q = \mathcal{R}^{-1}\GradG \phi \in L^2(\Lambda)$, we find
    that by definition
\begin{gather}    
    \begin{aligned}
      \sup_{ q\in V} \frac{b(q,\phi)}{\norm{\phi}{M}} & \geqslant
      \frac{ \| \mathcal{R}^{-1/2} \nabla_{\mathcal{G}} \phi \|^2_{L^2(\Lambda)}}{\norm{\phi}{M}} \geqslant
      \frac{ \| \mathcal{R}^{-1/2} \nabla_{\mathcal{G}} \phi \|^2_{L^2(\Lambda)}}{\norm{\mathcal{R}^{-1/2} \GradG \phi}{L^2(\Lambda)}} =
      \norm{\phi}{M}.
    \end{aligned}
    \end{gather}
This confirms the inf-sup condition with constant 1. 
\end{proof}

Applying standard Sobolev theory, we can show that the solution exhibits a higher regularity on each edge: 
\begin{theorem}[Higher regularity]
  Let $p \in M$, $q \in V$ solve \eqref{eq:varform-primal}. Then $\p \in H^2(\E)$ and $\q \in H^1(\E)$.
  \label{thm:higher-regularity}
\end{theorem}
\begin{proof}
    The proof is by a post-processing of the solution. On each edge $\Lambda_i$, define $\tilde{q}_i, \tilde{p}_i$ as the solutions of 
    \begin{align}
    \tilde{q}_i + \partial_s \tilde{p}_i &= f_\E \text{ on } \Lambda_i, \\
    \partial_s \tilde{q}_i &= p \text{ on } \partial \Lambda_i.
\end{align}
This problem is well defined as $p \in H^1(\Lambda)$, meaning that $p$ has a well-defined trace at the vertices. As $f_\E \in L^2(\Lambda_i)$, we further have $\tilde{p} \in H^2(\Lambda_i)$ and $\tilde{q}=H^1(\Lambda_i)$. By construction, $\tilde{p}_i=p$ and $\tilde{q}_i=q$ on each edge $\Lambda_i$, meaning that $p \in H^2(\E)$ and $q \in H^1(\E)$.
\end{proof}

\subsubsection{Stability of a family of primal discretizations}

Next, we propose to discretize the primal formulation~\eqref{eq:varform-primal} using $CG_k$ spaces for pressure and $DG_{k-1}$ spaces for the velocity defined relative to $\Lambda^h$ for $k \geq 1$, e.g.: 
\begin{equation}\label{eq:primal_h_spaces}
  M^h = CG_k(\Lambda^{h}), \quad V^h = DG_{k-1}(\Lambda^{h}).
\end{equation}
Assuming that $\mathcal{R}^{-1}$ is piecewise constant
these spaces satisfy the discrete Brezzi stability conditions.

\begin{remark}[Connection to finite volume schemes]
  Using the finite element spaces \eqref{eq:primal_h_spaces}, the primal mixed formulation
  can be interpreted as a staggered grid finite volume scheme~\cite{greyvenstein1994segregated}
  where the pressure and flux variables are bound to nodes and edges, respectively.
  The matrix representation of the discrete problem then takes the form 
  %
  \begin{equation}
  \begin{bmatrix}
    \pmb{\mathcal{R}} & \pmb{G}\\
    -\pmb{D}  & \pmb{0}\\       
  \end{bmatrix}
  \end{equation}
  where the matrix $\pmb{G}$, the discrete gradient/incidence matrix, encodes in its rows the connectivity of graph
  edges to nodes. Furthermore, 
  $\pmb{D}$ is the transpose of $\pmb{G}$ and $\pmb{\mathcal{R}}$ is a diagonal matrix of resistances for each edge of the graph. The Schur complement $-\pmb{D} \pmb{\mathcal{R}}^{-1} \pmb{G}$ is in fact the graph Laplacian, cf.~\Cref{rmrk:quantum_graph}.
\end{remark}

The norms \eqref{eq:primal_inf_sup_norms} induce an exact Schur complement preconditioner for the discretization by \eqref{eq:primal_h_spaces}. Since the stability constants in \Cref{thm:primal_inf_sup} are independent of both the graph geometry and the graph topology, we expect the condition number of \eqref{eq:eigenvalue} (i.e. $\xi^h_{\text{max}}/\xi^h_{\text{min}}$ with $\xi^h_{\min}$ and $\xi^h_{\max}$ denoting the smallest and largest in magnitude eigenvalues, respectively) to be constant for any $\mathcal{G}$ and mesh size. This theoretical expectation is confirmed by numerical experiments, see \Cref{tab:primal-infsup} for the case $\mathcal{R} = 1$ and the associated \Cref{fig:graph_generations}. We remark that these results would not be altered by varying the resistance.
\begin{figure}
  \begin{center}
  \begin{subfigure}{0.45\textwidth}
      \includegraphics[height=0.35\textwidth]{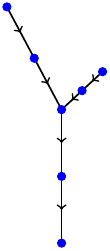}
      \includegraphics[height=0.35\textwidth]{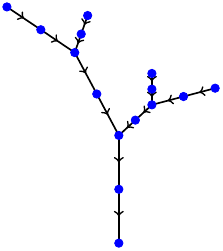}
      \includegraphics[height=0.35\textwidth]{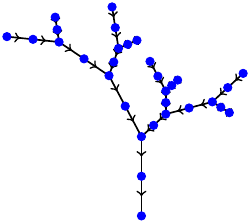}
      \caption{}
  \end{subfigure}
  \begin{subfigure}{0.45\textwidth}
    \includegraphics[height=0.35\textwidth]{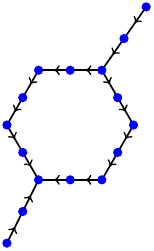}
    \includegraphics[height=0.35\textwidth]{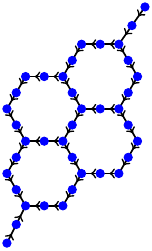}
    \includegraphics[height=0.35\textwidth]{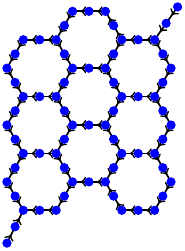}
    \caption{}
  \end{subfigure}
  \end{center}
  \caption{The (a) arterial \textit{tree} and (b) \textit{honeycomb} networks used for numerical experiments. From left to right, the networks grow by the addition of more edges. The arterial tree networks are grown by adding more generations; while the honeycomb networks grow by increasing the number of cycles.}
  \label{fig:graph_generations}
\end{figure}
\begin{table}
  \begin{subtable}{0.45\textwidth} \footnotesize
    \centering
    \begin{tabular}{l|cccc}
      \toprule
      \backslashbox{$n$}{$h$}
      &\makebox[2em]{1}&\makebox[2em]{0.5}&\makebox[2em]{0.25}
      &\makebox[2em]{0.125} \\
      \midrule
      1 & 2.618 & 2.618 & 2.618 & 2.618\\
      3 & 2.618 & 2.618 & 2.618 & 2.618\\
      7 & 2.618 & 2.618 & 2.618 & 2.618\\
      15 & 2.618 & 2.618 & 2.618 & 2.618\\
      \bottomrule
    \end{tabular}
    \caption{Arterial tree networks}
  \end{subtable}
  \hfill
  \begin{subtable}{0.45\textwidth} \footnotesize
    \centering
    \begin{tabular}{l|cccc}
      \toprule
      \backslashbox{$n$}{$h$}
      &\makebox[2em]{1}&\makebox[2em]{0.5}&\makebox[2em]{0.25}
      &\makebox[2em]{0.125}\\
      \midrule
      8 & 2.618 & 2.618 & 2.618 & 2.618\\
      18 & 2.618 & 2.618 & 2.618 & 2.618\\
      32 & 2.618 & 2.618 & 2.618 & 2.618\\
      50 & 2.618 & 2.618 & 2.618 & 2.618\\
      \bottomrule
    \end{tabular}
    \caption{Honeycomb networks}
  \end{subtable}
  \caption{
  Spectral condition numbers of the generalized eigenvalue problem  
  \eqref{eq:eigenvalue} with (referring to the notation introduced therein) $a$ being the primal mixed formulation and the norms on $V^h$, $M^h$ given by \eqref{eq:primal_inf_sup_norms}. Computations were performed on arterial tree and honeycomb networks (see \Cref{fig:graph_generations}), with $n$ denoting the number of bifurcation points, and $h$ denoting the mesh size.}
\label{tab:primal-infsup}
\end{table}

\subsection{Well-posedness of a dual mixed formulation} 
\label{sec:graph-analysis-dual}

We now turn to introduce and analyze a dual mixed formulation of the hydraulic network model~\eqref{eq:1d-hyd}. The Stokes--Brinkman network model~\eqref{eq:1d-stokes} can be expressed in a similar dual mixed form, as we also illustrate below but do not analyze further.  

\subsubsection{A dual mixed formulation of the network flow models}

To construct a dual mixed variational formulations of \eqref{eq:1d-hyd}, we multiply \eqref{eq:1d-hyd-1} by a test function $\psi \in H(\mathrm{div}; \mathcal{G})$ and \eqref{eq:1d-hyd-2} by a test function $\phi \in L^2(\mathcal{G})$ and integrate over $\mathcal{G}$. Multiplication in $L^2(\mathcal{G})$ implies that we multiply edge variables by edge variables, and vertex variables by vertex variables. We then find that the hydraulic network model can be expressed in the abstract form~\eqref{eq:abstract} with $V = H(\mathrm{div}, \mathcal{G})$ and $M = L^2(\mathcal{G})$ after defining
\begin{subequations}
  \begin{align}
    a(\q,\psi) &= (\q, \psi )_\E, \\
    b(\q, \phi) &= - (\DivG \q, \phi)_{\mathcal{G}}
    = - ( \partial_s \q, \phi)_{\mathcal{E}} - ( \jump{\q}, \phi)_{\mathcal{V}}, \\
    L(\psi) &= (g, \psi)_\E + (\tilde{p}^{\text{ref}}, \psi)_{\partial \V}, \\
    F(\phi) &= -(f , \phi )_{\mathcal{G}},
  \end{align}%
  \label{eq:a-to-F}%
\end{subequations}%
where the second term in $b$ accounts for the conservation of mass condition at the bifurcations. Further, $\tilde{p}^{\text{ref}}$ is given by the boundary conditions. We note that the Stokes--Brinkman network model~\eqref{eq:1d-stokes} can be expressed in an analogous dual mixed form over $V \times M$ with $b, L$ and $F$ given by~\eqref{eq:a-to-F}, and $a$ defined by
\begin{equation}
  a(\q,\psi)
  = (\GradG \nu_{\text{eff}} \q, \GradG \psi)_{\mathcal{E}} + (\q, \psi )_{\mathcal{E}} .
\end{equation}

The next result shows that the dual formulation~\eqref{eq:a-to-F} is well-posed. Moreover, its stability and inf-up constants defined relative to suitably weighted norms are uniform with respect to the graph topology and cardinality.
\begin{thm}
  Consider the dual mixed hydraulic network problem given by \eqref{eq:abstract} with \eqref{eq:a-to-F} defined over $V=H(\mathrm{div}; \mathcal{G})$ and $M=L^2(\mathcal{G})$ endowed with the weighted norms:
\begin{subequations}
  \label{eq:mixed_norm_C}
  \begin{align}
    \norm{\q}{V}^2 &= \norm{\q}{L^2(\mathcal{E})}^2 + \norm{ \ell \partial_s \q}{L^2(\mathcal{E})}^2 + \norm{\alpha^{-1} \ell \jump{q}}{L^2(\mathcal{V})}^2, \\
    \norm{\p}{M}^2& = \norm{\ell^{-1}\p_\mathcal{E}}{L^2(\mathcal{E})}^2 +  \norm{\alpha \ell^{-1}\p_\mathcal{V}}{L^2(\mathcal{V})}^2,
  \end{align}
\end{subequations}
where $\ell= \sum_{i=1}^n \ell_i$ is the total length of the network, and $\alpha$ is defined for each vertex $\vv_j$ as the square root of an averaged edge length:
\begin{equation}4
  \alpha_j^2 =  \left( \frac{\sum_{\Lambda_i \in E(\vv_j)} \ell_i}{m} \right) ,
  \label{eq:degree} 
\end{equation}
where $m$ is the total number of vertices in the network. Given $f \in L^2(\mathcal{G})$ and $g \in L^2(\E)$, there then exists a unique solution $(q, p) \in V \times M$. Moreover, the Brezzi coercivity and inf-sup constants are uniform with respect to the size and topology of the network.
\end{thm}
\begin{proof}
It is straightforward to show that the forms $a$ and $b$ are uniformly continuous with respect to the weighted norms on $V$ and $M$. Next, we show that $a$ is uniformly coercive on the kernel $K \subset V$ defined by:
\begin{align}
    K = \{  \psi \in V: b(\psi, p)&=0 \text{ for all } \p \in M  \}.
\end{align}
Consider any $\psi \in K$, and take $p_{\psi}=(p_\E, \p_{\V}) \in L^2(\mathcal{G})$ with $p_\E = \ell^2 \partial_s \psi$ and $p_\V=0$. A calculation then shows that
\begin{align}
    b(\psi, p_\psi)=(\nabla_\mathcal{G}\cdot \psi,p_\psi)_\mathcal{G} =  \norm{\ell \partial_s \psi}{L^2(\E)}^2 = 0 . 
\end{align}
Similarly, taking $p_\E = 0$ and $p_\V=l^2 \alpha^{-2} \jump{\psi}$ give that
\begin{align}
    b(\psi, p_\psi)=(\nabla_\mathcal{G}\cdot \psi,p_\psi)_\mathcal{G} = \norm{l \alpha^{-1} \jump{\psi}}{L^2(\V)}^{2} = 0 .
\end{align}
Thus,
\begin{align}
  a(\psi, \psi) = \norm{\psi}{L^2(\mathcal{E})}^2
  = \norm{\psi}{L^2(\mathcal{E})}^2 +  \norm{\ell \partial_s \psi}{L^2(\mathcal{E})}^2 + \norm{\ell \alpha^{-1} \jump{\psi}}{L^2(\mathcal{V})}^2
  = \norm{\psi}{V}^2 ,
\end{align}
and $a$ is uniformly coercive on $K$.

It remains to show that the inf-sup condition holds; i.e.,~that there exists a $\beta > 0$ such that
\begin{align}
  \sup_{\psi \in V} \frac{b(\psi,\p)}{\norm{\psi}{V}} \geq \beta \norm{\p}{M} \text{ for all } \p \in M .
  \label{eq:dual:inf-sup}
\end{align}
The proof is by construction of a suitable $\psi^p \in V$ so that
\begin{align}
  \norm{\psi^p}{V} \lesssim \norm{p}{M} 
  \label{eq:temp4}  
\end{align}
where we use $\lesssim$ to denote $\norm{\psi^p}{V} \leq C \norm{p}{L^2(\mathcal{G})}$ for some constant $C>0$ that is independent of the domain. To this end, fix $p \in M$. By Theorem \ref{thm:primal} and \ref{thm:higher-regularity}, there exists $\psi^p \in H^1(\E)$ and $\phi^p \in H^1(\Lambda) \cap H^2(\E)$ solving
\begin{gather}
  \begin{aligned}
    \psi^p + \nabla_\mathcal{G} \phi^p &= 0 &&\text{ on } \E, \\
    \nabla_\mathcal{G} \cdot \psi^p &= \ell^{-2} p &&\text{ on } \mathcal{G}.
  \end{aligned}
\end{gather}
To show \eqref{eq:temp4}, recall that
\begin{align}
  \norm{\psi^p}{V}^2  
  &= \norm{\psi^p }{L^2(\mathcal{E})}^2+
  l^{2} \norm{\partial_s \psi^p }{L^2(\mathcal{E})}^2 + l^{2} \norm{\alpha^{-1} \jump{\psi^p} }{L^2(\mathcal{V})}^2. \label{eq:tempV}
\end{align}
To bound the first term, note that for each edge $\Lambda_i$, there exists a $C_{s, i} > 0$, such that $\norm{\phi^p}{H^1(\Lambda_i)} \leq C_{s,i} \norm{\ell^{-2} p}{L^2(\Lambda_i)}$ for each edge $\Lambda_i$. Here, $C_{s, i}$ depend on the Poincar{\'e} constant $C_{p, i}$ of the domain $\Lambda_i$, and $C_{p, i} \sim \ell_i$~\cite{Arnold2009StabilityOL, kennedy2016spectral}. Thus 
\begin{gather}
    \begin{aligned}
      \norm{\psi^p }{L^2(\mathcal{E})}^2 &= \norm{\partial_{s} \phi^p }{L^2(\mathcal{E})}^2 \leq \norm{\phi^p }{H^1(\mathcal{E})}^2
      \leq  \sum_{i=1}^n C_{s,i}^2 \norm{\ell^{-2} p }{L^2(\Lambda_i)}^2 \\
      & \lesssim \sum_{i=1}^n \ell_i^2 \norm{\ell^{-2} p }{L^2(\Lambda_i)}^2
      \leq \norm{\ell^{-2} p }{L^2(\E)}^2 \sum_{i=1}^n \ell_i^2  \leq \ell^2 \norm{\ell^{-2} p }{L^2(\E)}^2 \\
      &= \norm{\ell^{-1} p }{L^2(\E)}^2.
      \label{eq:temp1}
    \end{aligned}
\end{gather}
The second term can be bounded by using that $\partial_s \psi^p = \ell^{-2} p$ edgewise:
\begin{align}
  \ell^{2} \norm{\partial_s \psi^p }{L^2(\mathcal{E})}^2 
  = \ell^{2} \norm{\ell^{-2} p }{L^2(\mathcal{E})}^2= \norm{\ell^{-1} p }{L^2(\mathcal{E})}^2.
  \label{eq:temp2}
\end{align}
To handle the third term, involving jumps of the solution across vertices, we use the trace inequality: there exists $C_{t, i} > 0$ such that
\begin{align}
  \norm{\psi^p_i}{L^2(\partial \Lambda_i)} \leq C_{t,i} \norm{\partial_s \psi^p_i}{L^2(\Lambda_i)} .
  \label{eq:trace-line}
\end{align}
The trace constant scales as $C_{t,i} \sim \ell_i^{-1/2}$. Thus
\begin{gather}
\begin{aligned}
  \ell^{2} \norm{\alpha^{-1} \jump{\psi^p} }{L^2(\mathcal{V})}^2 
            &\leq  \ell^{2} \sum_{\vv_j \in \V}  \alpha_j^{-2} 
  \sum_{\Lambda_i \in E(\vv_j)}\vert \psi^p_i(\vv_j)\vert^2  \\
            &\leq \ell^{2} \sum_{\vv_j \in \V}  \alpha_j^{-2} 
  \sum_{\Lambda_i \in E(\vv_j)} C_{t,i}^2 \norm{\partial_s \psi_i^p}{L^2(\Lambda_i)}^2  \\
            & \lesssim \ell^{2}\norm{\partial_s \psi^p}{L^2(\E)}^2 \sum_{\vv_j \in \V}  \alpha_j^{-2} 
 \sum_{\Lambda_i \in E(\vv_j)}  \ell_i^{-1}  \\
            &= \norm{\ell^{-1} p}{L^2(\E)}^2 \sum_{\vv_j \in \V}   \alpha_j^{-2} 
 \sum_{\Lambda_i \in E(\vv_j)} \ell_i^{-1}  \text{  (using \eqref{eq:temp2})} \\
            &= \norm{\ell^{-1} p}{L^2(\E)}^2  \sum_{\vv_j \in \V} \frac{1}{m}  \frac{\sum_{e_i \in E(\vv_j)} \ell_i}{\sum_{e_i \in E(\vv_j)} \ell_i} \\
            &= \norm{\ell^{-1}p}{L^2(\E)}^2.
            \label{eq:temp3}
\end{aligned}
\end{gather}
Combining \eqref{eq:temp1}--\eqref{eq:temp3} then gives \eqref{eq:temp4}, and thus we find that
\begin{align}
  \sup_{\psi \in V} \frac{b(\psi,\p)}{\norm{\psi}{V}} \geq \frac{b(\psi^p,\p)}{\norm{\psi^p}{V}}
  = \frac{(\nabla_\mathcal{G} \cdot \psi^p,\p)}{\norm{\psi^p}{V}}
  =\frac{(\ell^{-1} \p, \ell^{-1}\p)_{L^2(\mathcal{G})}}{\norm{\psi^p}{V}}=\frac{\norm{\p}{M}^2}{\norm{\psi^p}{V}} \geq \beta \norm{\p}{M}
\end{align}
and the inf-sup condition~\eqref{eq:dual:inf-sup} holds.
\end{proof}

\subsubsection{Stability and robustness of families of dual discretizations}
Next, we consider finite element discretizations $V^h \times M^h \subset V \times M$ of the dual mixed hydraulic network model given by \eqref{eq:abstract} with \eqref{eq:a-to-F}. We let
\begin{subequations}
  \label{eq:mixed_spaces}
  \begin{align}
    V^h &= CG_k(\Lambda_1^h) \times CG_k(\Lambda_2^h) \times \dots \times CG_k(\Lambda_n^h), \\
    M^h &= DG_{k-1}(\Lambda^h) \times \mathbb{R}^m ,
  \end{align}
\end{subequations}
which correspond to branch-wise Raviart--Thomas-type spaces with the flux glued together using Lagrange multipliers at the internal vertices of the network. For the dual 
Stokes--Brinkman network model, we consider the same $V^h$ but instead consider continuous pressures; i.e., the pairing
\begin{subequations}
  \label{eq:mixed_spaces_TH}
  \begin{align}
    V^h &= CG_k(\Lambda_1^h) \times CG_k(\Lambda_2^h) \times \dots \times CG_k(\Lambda_n^h), \\
    M^h &= CG_{k-1}(\Lambda^h) \times \mathbb{R}^m,
  \end{align}
\end{subequations}
which correspond to branch-wise Taylor--Hood-type elements. 

\begin{remark}
  The dual mixed hydraulic network model constitutes the same variational formulation as was derived by \citet{cerroni2018mixed}. Therein, the vertex values of $p$ were introduced as Lagrange multipliers enforcing conservation of mass, and the system was discretized using branch-wise Taylor--Hood elements. However, our numerical experiments indicate that this choice is not inf-sup stable. Instead, we find that the hydraulic network model should be discretized using branch-wise Raviart--Thomas, while the dual mixed Stokes--Brinkman formulation could be discretized using branch-wise Taylor--Hood.
\end{remark}

\begin{remark}[Connection to non-conforming hybridized methods]
    We note that with \eqref{eq:mixed_spaces} the discrete dual mixed formulation is closely
  related to the non-conforming (hybridized) mixed methods for the Darcy equation \cite{Arnold1985MixedAN} where
  element-local $H(\text{div})$ spaces are glued across facets by Lagrange multipliers. Applying these ideas to our network setting, where the roles of elements/cells and facets are played respectively by graph edges 
  and vertices, yields the norms (cf. \ref{eq:mixed_norm_C})
  \begin{subequations}
    \label{eq:mixed_norm}
    \begin{align}
      \norm{q}{V^h}^2 &= \norm{q}{L^2(\mathcal{E})}^2 + \norm{\partial_s q}{L^2(\mathcal{E})}^2 + \sum_{j=1}^m h^{-1}_j\jump{q}_j^2, \\
      \norm{p}{M^h}^2 &= \norm{p_{\mathcal{E}}}{L^2(\mathcal{E})}^2 + \sum_{j=1}^m h_j p_{\mathcal{V}}^2.
    \end{align}%
  \end{subequations}%
  Here, for any internal vertex $b$, $h_j$ denotes the mean length of finite element cells in $\Lambda_i^h$
  connected to $b$. Thus $h_j$ depends on the mesh and the degree of the node. 
\end{remark}

Now, we  examine numerically the robustness and conditioning of the dual discretization. Using lowest order elements $k=1$ in the family of discretizations \eqref{eq:mixed_spaces}, 
\Cref{tab:dual_simple}
and \Cref{tab:dual-infsup} report respectively the condition numbers of the dual mixed formulation using the unweighted norms (in particular 
the $V$ norm \eqref{eq:simple_hdiv_norm}) and the domain-dependent norms \eqref{eq:mixed_norm_C}. In both cases the condition numbers appear to be stable in $h$, however, only the weighted norms lead to boundedness also in the number of bifurcations for different graph configurations (tree, honeycomb). Note that the length of the graph $\ell$ increases with the number of generations in these graph configurations. For the honeycomb networks, $\ell$ grows from approximately $\ell=6$ to $\ell=16$ between the first and final generations, while for the tree graphs $17 < \ell < 50$.

Let us finally comment on robustness and stability with respect to the resistance parameter. This is of particular interest for simulations of flow in branching networks, where the cross-section size typically reduces at each branching generation. In this case, we may apply results from $\mathcal{R}$-robust Darcy preconditioners \cite{badia2010stabilized} to propose the following norms for the solution spaces (instead of~\eqref{eq:mixed_norm_C}): 
\begin{subequations}
  \label{eq:mixed_norm_R}
    \begin{align}
      \lVert q \rVert^2_{V} &= \lVert \mathcal{R}^{1/2} q \rVert^2_{L^2(\mathcal{E})}
      + \lVert \ell \mathcal{R}^{1/2} \partial_s q \rVert^2_{L^2(\mathcal{E})}
      + \norm{ \ell \mathcal{R}_j^{1/2} \alpha^{-1}\jump{q}}{L^2(\V)}^2 ,\\
      \lVert p \rVert^2_{M} &= \lVert \ell^{-1} \mathcal{R}^{-1/2} p_{\mathcal{E}} \rVert^2_{L^2(\mathcal{E})}
      +  \norm{\ell^{-1} \alpha \mathcal{R}_j^{-1/2} p_{\V}}{L^2(\V)}^2.
    \end{align}
\end{subequations}
Here $\mathcal{R}_j$ represents the mean resistance at bifurcation point $\vv_j$ defined by averaging over connected branches. The robustness of the mixed formulation with norms \eqref{eq:mixed_norm_R} is demonstrated numerically in \Cref{tab:dual-infsup-kappa}.

\begin{table}
  \begin{subtable}{0.45\textwidth}
    \footnotesize
    \centering
    \begin{tabular}{l|*{4}{c}}
      \toprule
\backslashbox{$n$}{$h$}
&\makebox[2em]{1}&\makebox[2em]{0.5}&\makebox[2em]{0.25}
&\makebox[2em]{0.125} \\
\midrule
%
1  & 29.659 & 34.316 & 36.349 & 36.959\\
3  & 51.474 & 56.193 & 57.820 & 58.306\\
7  & 68.691 & 73.171 & 74.593 & 74.984\\
15 & 81.268 & 85.209 & 86.381 & 86.691\\
28 & 89.779 & 93.219 & 94.197 & 94.450\\
50 & 93.619 & 96.669 & 97.525 & 97.745\\
  \bottomrule
\end{tabular}
\caption{Tree networks}
\end{subtable}
\caption{Condition numbers for the dual mixed discretizations with 
unweighted norms of $V\times M$, i.e. $l=1$, $\alpha=1$ 
in \eqref{eq:mixed_norm_C}. Preconditioning based on these 
norms yields linear systems which become stiffer as the network length and complexity grow.
}
\label{tab:dual_simple}
\end{table}

\begin{table}
  \begin{subtable}{0.45\textwidth}
    \footnotesize
    \centering
    \begin{tabular}{l|*{4}{c}}
      \toprule
\backslashbox{$n$}{$h$}
&\makebox[2em]{1}&\makebox[2em]{0.5}&\makebox[2em]{0.25}
&\makebox[2em]{0.125} \\
\midrule
%
  1   &    2.679 &    3.914 &    3.613 &    3.997 \\   
  3   &    2.698 &    3.935 &    3.639 &    4.014 \\  
  7   &    2.704 &    3.942 &    3.648 &    4.02  \\  
  15  &    2.706 &    3.944 &    3.650  &    4.022 \\  
  28  &    3.877 &    4.123 &    4.091 &    4.126 \\  
  50  &    4.034 &    4.122 &    4.110  &    4.125 \\  
  108 &    4.033 &    4.127 &    4.109 &    4.128 \\  
  \bottomrule
\end{tabular}
\caption{Tree networks}
\end{subtable}
  \hspace{3em}
  \hfill
  \begin{subtable}{0.45\textwidth}\footnotesize
\centering
\begin{tabular}{l|*{4}{c}}
  \toprule
\backslashbox{$n$}{$h$}
&\makebox[2em]{1}&\makebox[2em]{0.5}&\makebox[2em]{0.25}
&\makebox[2em]{0.125} \\
\midrule
%
  8    &3.659   &3.990   &3.909   &4.021\\
  18   &3.729   &3.997   &3.932   &4.026\\
  32   &3.775   &4.009   &3.947   &4.031\\
  50   &3.812   &4.013   &3.961   &4.032\\
  72   &3.815   &4.018   &3.969   &4.035\\
  98   &3.845   &4.020   &3.980   &4.036\\
  128  &3.839   &4.022   &3.981   &4.038\\
  \bottomrule
\end{tabular}
\caption{Honeycomb networks}
\end{subtable}
\caption{Condition numbers for the dual mixed discretizations with norms given by \eqref{eq:mixed_norm_C}. Computations were performed on tree networks (a) and honeycomb networks (b), with $n$ denoting the number of internal graph vertices and $h$ denoting the mesh size. The resistance parameter was set to $\mathcal{R}=1$.}
\label{tab:dual-infsup}
\end{table}

\begin{table}
\begin{subtable}{0.75\textwidth} \footnotesize 
\begin{tabular}{l|*{7}{c}}\toprule
\backslashbox{$n$}{$\mathcal{R}$}
&\makebox[3em]{$10^{-8}$}&\makebox[3em]{$10^{-4}$}&\makebox[3em]{$10^{-2}$}&\makebox[3em]{$1$}&\makebox[3em]{$10^{2}$}&\makebox[3em]{$10^{4}$}&\makebox[3em]{$10^{8}$}\\\midrule
  1  &     2.679 &       2.679 &     2.679 &    2.679 &      2.679 &        2.679 &            2.679 \\
  3  &     2.698 &       2.698 &     2.698 &    2.698 &      2.698 &        2.698 &            2.698 \\
  7  &     2.704 &       2.704 &     2.704 &    2.704 &      2.704 &        2.704 &            2.704 \\
  15 &     2.706 &       2.706 &     2.706 &    2.706 &      2.706 &        2.706 &            2.706 \\
  28 &     3.877 &       3.877 &     3.877 &    3.877 &      3.877 &        3.877 &            3.877 \\
  50 &     4.034 &       4.034 &     4.034 &    4.034 &      4.034 &        4.034 &            4.034 \\
  108&     4.033 &       4.033 &     4.033 &    4.033 &      4.033 &        4.033 &            4.033 \\
  \bottomrule
\end{tabular}
  \end{subtable}
\caption{
  Conditioning of the dual mixed formulation with norm \eqref{eq:mixed_norm_R} for different values of resistance
  parameter (constant in space). Tree networks are considered with $n$ denoting the number of internal graph vertices.
  The mesh size is fixed at $h=1$.
}
\label{tab:dual-infsup-kappa}
\end{table}

\subsection{Approximation and convergence of primal and dual discretizations}
\label{sec:graph-conv}

To examine the approximation properties of the primal and dual mixed hydraulic network models, we compute the error and convergence rates against an analytic solution of  a simple bifurcation problem. To be more precise, let $\vv_1=(0,0)$, $\vv_1=(0,0.5)$, $\vv_2=(-0.5,1)$ and $\vv_3=(0.5,1)$. From these vertices, we create a Y-shaped (bifurcating) graph by setting $e_1=(\vv_1, \vv_2)$, $e_2=(\vv_2, \vv_3)$ and $e_2=(\vv_2, \vv_3)$. Each edge is associated with a resistance $\mathcal{R}=1$ and a cross-section area $A=1$. Letting $s$ denote the distance from the root node $\vv_1$, we take
\begin{align*}
  q &= 
  \begin{cases}
    1 + \cos(\pi s) +\sin(2\pi s) & \text{ on } e_1,\\
    \frac{1}{2} + \cos(\pi s) +\sin(2\pi s) & \text{ on } e_2, e_3,
  \end{cases} \\
  p &= \sin(\pi s)+\cos(2\pi s) \text{ on } e_1,e_2,e_3 . 
\end{align*}
as analytic solutions; inserting these in \eqref{eq:1d-hyd} gives the associated values for $f$ and $g$. Finally, we use the analytic solution pressure to impose suitable pressure boundary conditions. We note that $p$ is smooth on all of $\Lambda$. Contrarily, $q$ is smooth on all edges $\Lambda_i$, but discontinuous across the bifurcation. 

Table \ref{tab:spatial_conv} shows the errors and convergence rates associated with these discrete solutions. The primal mixed approximation shows order $k$ convergence for $k = 1, 2, 3$ for both flux and pressure, measured in the $L^2(\Lambda)$- and $H^1(\Lambda)$-norms, respectively. This agrees with the expected rates for standard finite element methods \cite{brenner}.

For $k=1$, the dual mixed approximation similarly shows order one convergence of the pressure and flux, now measured in the $L^2$- and $H(\mathrm{div}; \mathcal{G})$-norms, respectively.  Increasing the degree, we find that the flux approximation enjoys $k$-order convergence in the $H(\mathrm{div}; \mathcal{G})$-norms, while the pressure error converges at a maximum rate of two. \citet{egger2023hybrid}[Lemma 4] showed that higher order convergence for a similar numerical method is possible in a single vessel. Indeed, repeating the convergence test on a single vessel (no bifurcations), we found that the optimal $k$-order convergence was restored. The lack of higher-order convergence is therefore likely a consequence of the bifurcation condition.

\begin{table}
\small
\centering
\begin{subtable}{0.6\textwidth}
\centering

\footnotesize $k=1:$ \hspace{1em}
\begin{subtable}{0.6\textwidth} \footnotesize
\centering
\begin{tabular}{l | l l}
$h$     & $\norm{q}{L^2(\Lambda)}$  &$\norm{p}{H^1(\Lambda)}$ \\ \hline
3.1e-02 &        1.5e-01         &               1.8e+00        \\
1.6e-02 &        7.2e-02   (1.0) &               9.0e-01   (1.0) \\
7.8e-03 &        3.6e-02   (1.0) &               4.5e-01   (1.0) \\
3.9e-03 &        1.8e-02   (1.0) &               2.2e-01   (1.0) 
\end{tabular}
\end{subtable}
\vspace{2em}

\footnotesize $k=2:$ \hspace{1em}\begin{subtable}{0.6\textwidth} \footnotesize
\begin{tabular}{l | l l}
$h$     & $\norm{q}{L^2(\Lambda)}$ & $\norm{p}{H^1(\Lambda)}$ \\ \hline
3.1e-02 &        1.0e-02         &               1.3e-01         \\
1.6e-02 &        2.6e-03   (2.0) &               3.2e-02   (2.0) \\
7.8e-03 &        6.4e-04   (2.0) &               8.1e-03   (2.0) \\
3.9e-03 &        1.6e-04   (2.0) &               2.0e-03   (2.0) 
\end{tabular}
\end{subtable}
\vspace{2em}

\footnotesize $k=3:$ \hspace{1em} 
\begin{subtable}{0.6\textwidth} \footnotesize
\begin{tabular}{l | l l}
$h$     & $\norm{q}{L^2(\Lambda)}$ &$\norm{p}{H^1(\Lambda)}$ \\ \hline
3.1e-02 &        4.5e-04         &               9.3e-03         \\
1.6e-02 &        5.7e-05   (3.0) &               1.2e-03   (3.0) \\
7.8e-03 &        7.1e-06   (3.0) &               1.5e-04   (3.0) \\
3.9e-03 &        8.8e-07   (3.0) &               1.9e-05   (3.0) 
\end{tabular}
\end{subtable}

\caption{Primal mixed discretization}
\label{fig:spatial_conv_primal}
\end{subtable}
\begin{subtable}{0.39\textwidth}
\centering
\begin{subtable}{0.9\textwidth} \footnotesize
\centering
\begin{tabular}{l | l  l }
$h$    & $\norm{q}{H(\text{div};\mathcal{G})}$ &$\norm{p}{L^2(\mathcal{G})}$ \\ \hline
3.1e-02 &     8.6e-01         &       7.4e-02   (0.0) \\
1.6e-02 &     4.3e-01   (1.0) &       3.5e-02   (1.1) \\
7.8e-03 &     2.2e-01   (1.0) &       1.7e-02   (1.0) \\
3.9e-03 &     1.1e-01   (1.0) &       8.6e-03   (1.0) 
\end{tabular}
\end{subtable}
\vspace{2em}

\begin{subtable}{0.9\textwidth} \footnotesize
\begin{tabular}{l | l l }
$h$    & $\norm{q}{H(\text{div};\mathcal{G})}$ &$\norm{p}{L^2(\mathcal{G})}$ \\ \hline
3.1e-02 &     2.8e-02         &       6.6e-03         \\
1.6e-02 &     7.0e-03   (2.0) &       1.6e-03   (2.0) \\
7.8e-03 &     1.8e-03   (2.0) &       4.1e-04   (2.0) \\
3.9e-03 &     4.4e-04   (2.0) &       1.0e-04   (2.0) \\
\end{tabular}
\end{subtable}
\vspace{2em}

\begin{subtable}{0.9\textwidth} \footnotesize\begin{tabular}{l |  l l }
$h$     & $\norm{q}{H(\text{div};\mathcal{G})}$ &$\norm{p}{L^2(\mathcal{G})}$ \\ \hline
3.1e-02 &     3.8e-04         &       4.3e-03         \\
1.6e-02 &     4.7e-05   (3.0) &       1.1e-03   (2.0) \\
7.8e-03 &     5.9e-06   (3.0) &       2.6e-04   (2.0) \\
3.9e-03 &     7.4e-07   (3.0) &       6.6e-05   (2.0) 
\end{tabular}
\end{subtable}
\caption{Dual mixed discretization}
\label{fig:spatial_conv_dual}
\end{subtable}
\caption{Approximation errors (and convergence rates) of the (left) primal mixed discretization~\eqref{eq:varform-primal} and (right) dual mixed hydraulic network discretization~\eqref{eq:a-to-F} for an idealized test case, with discrete spaces given by \eqref{eq:primal_h_spaces} and \eqref{eq:mixed_spaces}, respectively, and for $k = 1, 2, 3$. We observe optimal convergence orders for the approximation of the velocity $q$ for both formulations and all $k$, and for the pressure $p$ for both formulations for $k = 1, 2$. For $k = 3$, the primal formulation yields optimal rates also for $p$, while the dual formulation is one order suboptimal.}
\label{tab:spatial_conv}
\end{table}

\section{Flow in physiological perivascular networks}
\label{sec:pvs-results}

We now turn to simulate CSF flow within physiologically meaningful vascular networks. We first address the question of whether (infinitely) long wavelength pulsations of the vascular wall can induce directional net flow in non-trivial perivascular networks (Section~\ref{sec:modelling-sim-pumping}). Next, we study flow in image-based cortical vascular networks extending as a continuous, closed space from the arterial to the venous sides (Section \ref{sec:micro}). This configuration, with a minimum of CSF influx and efflux routes, produces purely oscillatory perivascular flow. Indicates that connection routes between the PVSs and the surrounding tissue are vital in producing directional flow. 
 
For the computational experiments, we consider the time-dependent formulation of the hydraulic network model~\eqref{eq:1d-hyd} to numerically compute the cross-section flux $q$ and pressure $p$. We employ a first-order (implicit Euler) discretization of the time-derivative and a finite element discretization of the analogous primal formulation cf.~\eqref{eq:varform-primal} with the lowest-order discrete spaces given by \eqref{eq:primal_h_spaces} ($k = 1$). The mesh refinement, number of time steps, and number of cycles were increased until the reported numbers were accurate to the first digit. 

\subsection{Can long-wavelength pulsations induce directional flow in perivascular networks?}
\label{sec:modelling-sim-pumping}

Consider a synthetic network of bifurcating blood vessels and surrounding PVSs represented by a graph $\mathcal{G}$. We assume that the network includes $N_{\rm gen}$ generations and at baseline obeys Murray's law; i.e., that the blood vessel radii at each junction satisfy the relation
\begin{equation}
  (R^1_p)^3 = (R^1_{d_1})^3+(R^1_{d_2})^3,
  \label{eq:murray}
\end{equation}
where $R^1_{p}$ and $R^1_{d_1}, R^1_{d_2}$ denote the baseline inner radii of
the parent and two daughter vessels, respectively. To quantify the symmetry of the network, we introduce the branching radius symmetry $\gamma=R^1_{d_1}/R^1_{d_2}$. Each vessel $\Lambda_i$ is scaled such that $\ell_i=10 R^1_{i}$. We assume that the PVS cross-sections are annular, with inner radius $R^1_i$ and outer radius $R^2_i = 3 R^1_i$ at baseline. Moreover, we model the PVS as non-porous ($\varphi=1$) and filled with CSF with viscosity $\nu = 1 \cdot 10^{-6}$m$^2$/s as of water. We set the root vessel radius $R^1_0 = 1$mm.

We model vascular contractions and expansions by prescribing the motion of the inner vascular wall, leaving the outer PVS boundary fixed. To isolate the effect of PVS network structure, we consider vascular pulsations in the form of uniform waves; that is, simultaneous expansions (or contractions) of the inner wall segments by 
\begin{align}
  R^1(s,t) = \left( 1 + \epsilon \sin \left( \frac{t}{T_{\text{cycle}}}\right) \right) R^1(s,0) 
  \label{eq:vasomotion}
\end{align}
with amplitude $\epsilon = 0.1$ and frequency $T_{\text{cycle}}^{-1}=1$ Hz, the latter corresponding to cardiac-induced arterial pulsations~\citep{mestre2018flow}. These changes in the inner radius will push CSF out of (or into) the PVSs. We here allow for CSF to flow freely into the tissue via the PVS inlets and outlets by setting the tissue pressure at a reference pressure ($\tilde{p}^{\text{ref}}=0$) at all boundary vertices. Hence, we tacitly assume that CSF flow from PVS into tissue does not alter tissue pressure. Additionally, recall that changes to the inner radius will change the size of the cross-section, and hence alter the resistance field as per \eqref{eq:resistance}. We initialize the system at rest, $q(0)=0$.

We are interested in quantifying the net flow within the PVS network. The directional net flow $Q(s; t_1, t_2)$ through a point $\boldsymbol{\lambda}(s)$ between the times $t_1$ and $t_2$, and its cycle-average net flow rate $\langle Q(s) \rangle$ are defined by
\begin{equation}
    Q(s; t_2, t_1) = \int_{t_1}^{t_2} q (s, t) \, \mathrm{d}t, \quad
    \langle Q(s) \rangle = \frac{1}{T_\text{cycle}} \int_{t'}^{t'+T_\text{cycle}} q(s, t) \, \mathrm{d}t,
\end{equation}
where $T_\text{cycle}$ denotes pulsation period (time for one cycle) and $t' > 0$ is arbitrary. Naturally, the volume of fluid being displaced depends on the amplitude $\epsilon$ and the length of the vessel $\ell_i$. To measure the directionality of the displaced flow, we split $q$ into oscillatory and directional parts,
\begin{align}
   q = q_{\text{osc}} +q_{\text{dir}},
\end{align}
where $q_{\text{osc}}$ is defined so that its associated net flow $\langle Q_{\text{osc}} \rangle = 0$. Next, we 
define the directionality ratio
\begin{align}
    \eta = \frac{\langle Q(s)\rangle }{\text{max}(q_{\text{osc}})}, 
    \label{eq:directionality}
\end{align}
where $\text{max}(q_{\text{osc}})$ denotes the oscillatory amplitude.

Interestingly, these uniform waves induce oscillatory and directional flow in the Murray networks with more than one generation. Figure \ref{fig:vaso} illustrates this phenomenon in an arterial tree consisting of five generations with $\gamma = 1$. In this visual representation, we have tracked the net flow at the root node and two leaf nodes. As shown, the flow exhibits both oscillatory and directional characteristics. CSF enters via the root node and exits through leaf nodes; with one exception: the node marked in green also sees an influx of CSF. 

The directional net flow depends on the perivascular network configuration (Table \ref{tab:netflow-N-gen-tree}). For a one-generation network ($N_{\rm gen} = 1$, a single vessel), the flow at the inlet remains entirely pulsatile with zero average net flow. However, the net directional flow increases with the number of network generations. Mostly, but not always, networks with larger aspect ratios ($\gamma \ll 1$) admit less net flow. We remark that as $\gamma$ is decreased, the network increasingly resembles a single vessel (with a constant radius), a case in which uniform waves do not induce non-zero net flow. Similar experiments in networks with homogeneous radii (i.e. where all vessels are assigned the same initial radius rather than by Murray's law) yield negligible net flow.  
\begin{figure}
\includegraphics[width=0.6\textwidth]{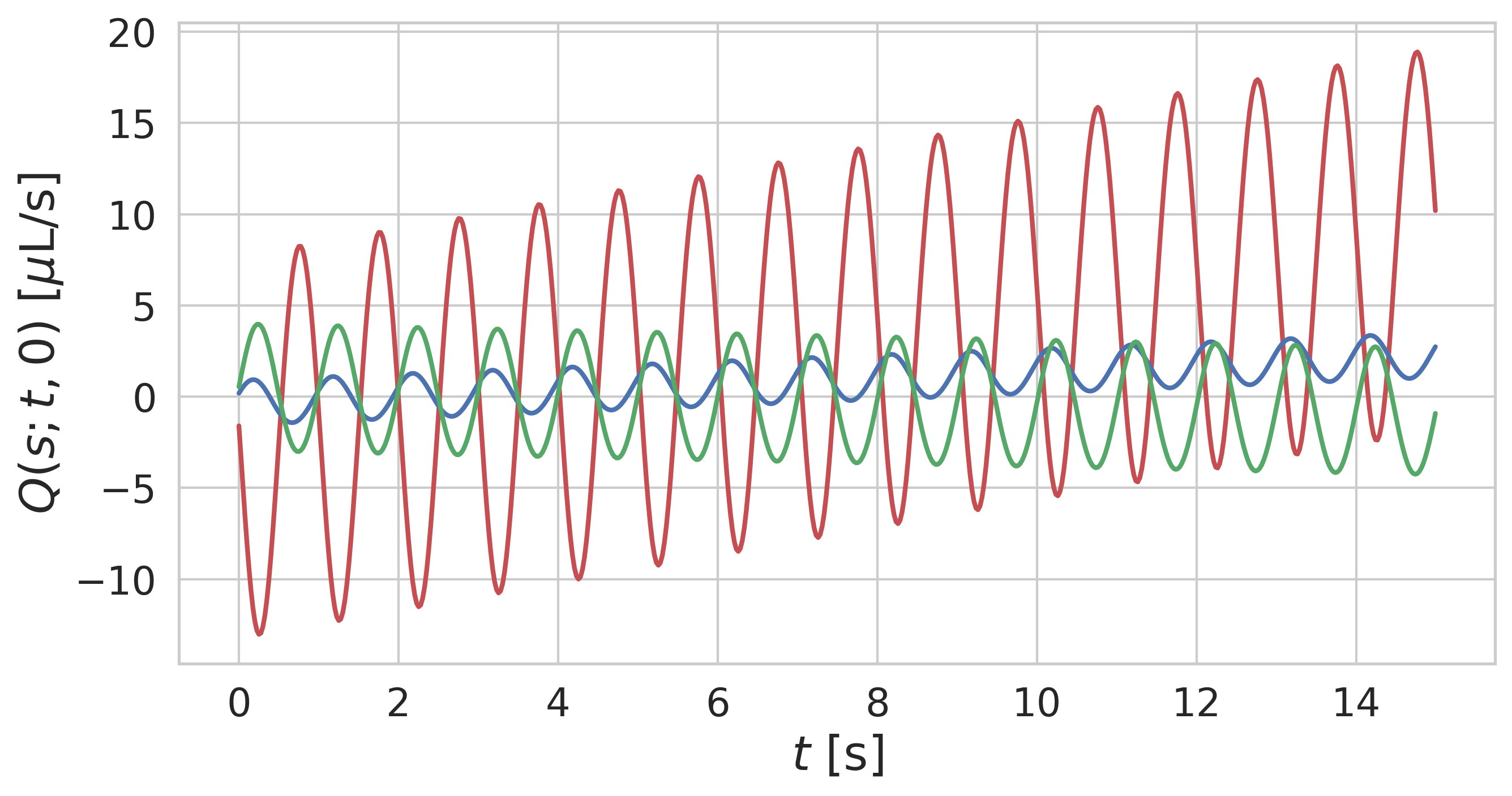}
\begin{tikzpicture}
    \node[anchor=south west, inner sep=0] (image) at (0,0) {\includegraphics[width=0.3\textwidth]{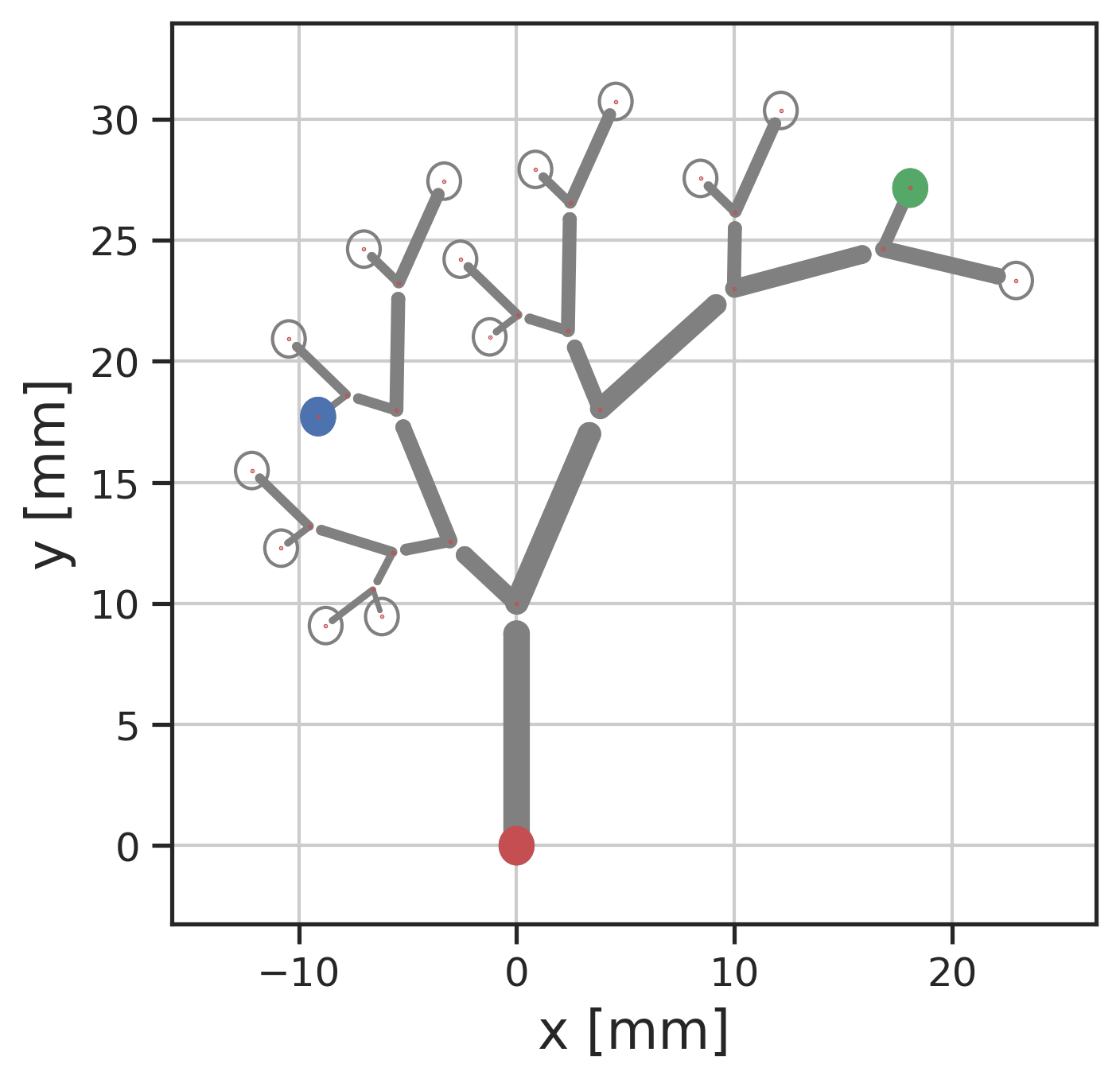}};
        \draw[gray] (1,4) circle (0.08cm); 
        \node[draw=none, text=black] at (2.2,4) {\scriptsize{Inlet/outlets}};
    \end{tikzpicture}
\caption{Net flow $Q(s;t,0)$ (left) over time due to uniform wave pulsations in an aterial pulsations in a five-generational arterial tree (right). The tree has open inlets/outlets, and net flow is tracked through the inlet note (red) and two outlets (blue, green). Arterial pulsation can be seen to drive both oscillatory and directional flow.}
\label{fig:vaso}
\end{figure}

\begin{table}
\centering
\begin{subtable}{0.99\textwidth}
\hspace{10em}
\includegraphics[width=0.1\textwidth,trim=50 60 60 60,clip]{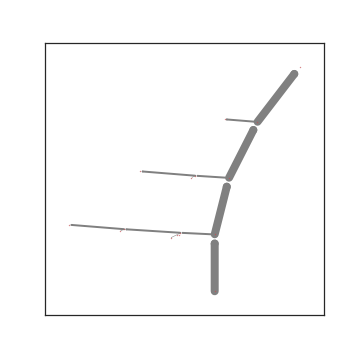} \hspace{3em}
\includegraphics[width=0.1\textwidth,trim=50 60 60 60,clip]{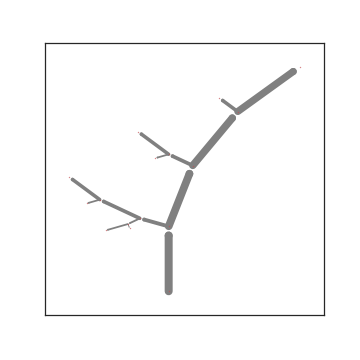}\hspace{3em}
\includegraphics[width=0.1\textwidth,trim=50 60 60 60,clip]{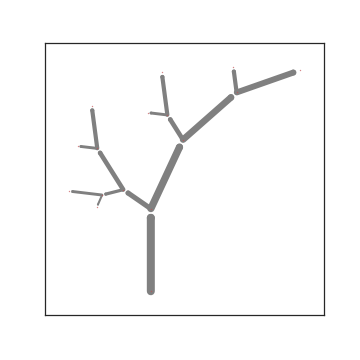}\hspace{3em}
\includegraphics[width=0.1\textwidth,trim=50 60 60 60,clip]{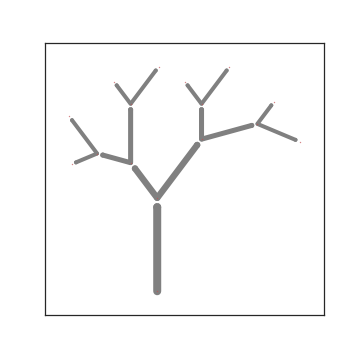}
\hspace{5em}
\end{subtable}
\vspace{0.1em}
\begin{tabular}{|l|*{4}{l}}
\toprule 
\backslashbox{$N_{\text{gen}}$}{$\langle Q \rangle $ ($\eta$)}
&\makebox[3em]{$\gamma=\nicefrac{1}{4}$}&\makebox[3em]{$\gamma=\nicefrac{1}{2}$}&\makebox[3em]{$\gamma=\nicefrac{3}{4}$}&\makebox[3em]{$\gamma=1$} \\
\midrule
       1  &  0.00$\mu$L (0.0\%) & 0.00 $\mu$L (0.0\%) &  0.00$\mu$L (0.0\%) & 0.00$\mu$L (0.0\%) \\
       2  &  0.05$\mu$L (0.4\%) & 0.14$\mu$L (1.3\%) &  0.11$\mu$L (1.1\%) & 0.08$\mu$L (0.9\%) \\
       3  &  0.07$\mu$L (0.4\%) & 0.30$\mu$L (1.8\%) &  0.30$\mu$L (2.1\%) & 0.24$\mu$L (1.8\%) \\
       4  &  0.10$\mu$L (0.4\%) & 0.51$\mu$L (2.4\%) &  0.57$\mu$L (3.1\%) & 0.45$\mu$L (2.8\%)  \\
       5  &  0.14$\mu$L (0.5\%) & 0.71$\mu$L (2.6\%) &  0.80$\mu$L (3.7\%) & 0.81$\mu$L (4.0\%) *  \\
       6  &      0.20$\mu$L (0.5\%)         & 1.00$\mu$L (3.1\%) &  1.20$\mu$L (4.5\%) & 1.20$\mu$L (5.0\%)  \\ 
       \bottomrule
    \end{tabular}    
     \caption{Average net flow and directionality $\eta$ (in parenthesis) at the inlet node induced by uniform inner wall waves in perivascular trees. Simulation setup is as in Figure \ref{fig:vaso}. The branching symmetry parameter $\gamma$ strongly affects the amount of net flow. The entry marked with an asterisk corresponds to the simulation visualized in Figure \ref{fig:vaso}.}
     \label{tab:netflow-N-gen-tree}
\end{table}

\subsection{Net flow is dependent on PVS-tissue connection, and does not occur for continuous arterial-capillary-venous PVSs}
\label{sec:micro}

We now turn to consider an image-based network extracted from a 1mm$^3$ cube of cortical tissue~\cite{goirand2021network, blinder2013cortical} (Figure \ref{fig:micro}). The network is described by the spatial locations and radius of $\sim$15,000 vessels/edges and includes 918 arteries, 216 veins, and 12559 capillaries. The perivascular space is modelled as a continuous space extending from arteries to veins. Arterial vessels are assigned uniform pulsations \eqref{eq:vasomotion} with amplitude $\epsilon = 0.1$ and frequency $T_{\text{cycle}}^{-1}=1$ Hz. Capillaries and veins are assumed to have fixed radii in time.

For this configuration, the simulated flow is purely pulsatile measured at arterial inlets (red curves in Figure~\ref{fig:micro-results}). This is in contrast to the results in the previous section (cf.~Figure \ref{fig:vaso}), where net flow was found to occur in a network with open inlets and outlets. We hypothesize that the minuscule cross-sections of the capillary PVS connecting the arterial and venous sides play an important role for these observations. These give rise to a high resistance in the capillary part of the network. Thus, while this network is endowed with multiple inlets and outlets, the capillary PVS effectively act as a no-flow zone, thus eliminating the route for net fluid movement. Indeed, negligible net flow is observed at the venous outlets (green curves in Figure~\ref{fig:micro-results}). 
\begin{figure}
\label{fig:micro-example}
\begin{subfigure}{0.3\textwidth}
  \includegraphics[width=0.95\textwidth]{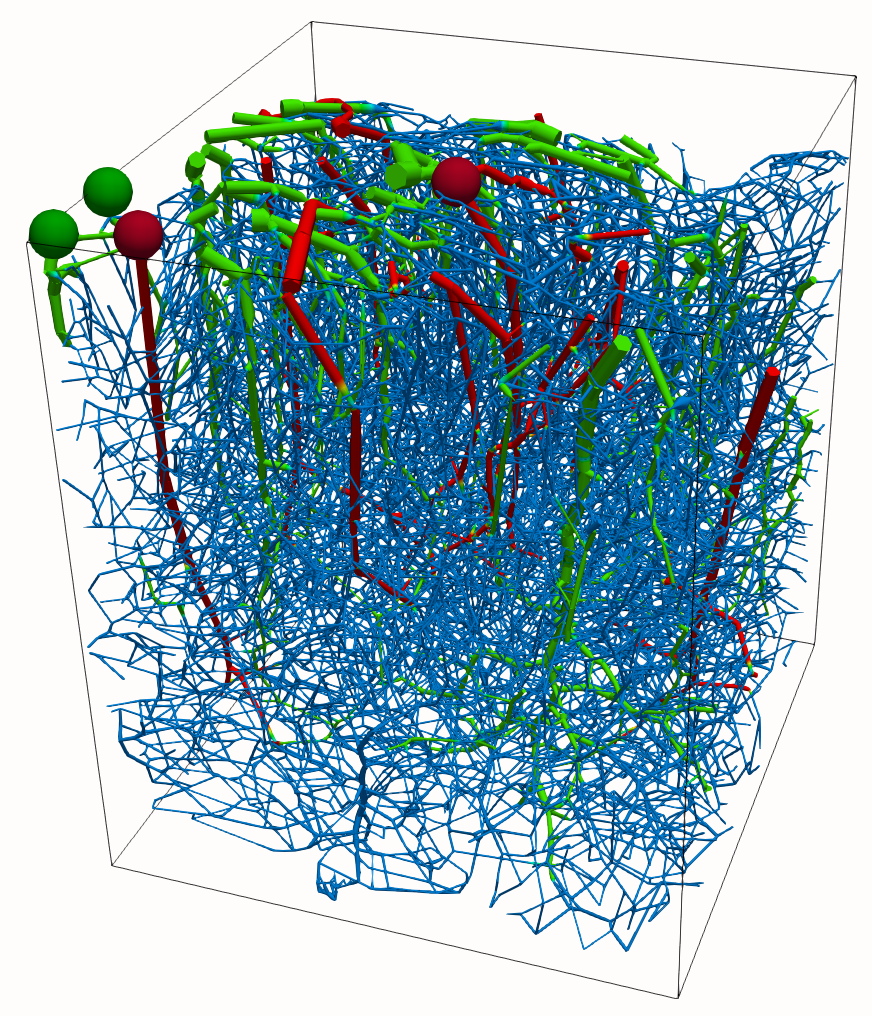}
  \caption{}
  \label{fig:micro}
\end{subfigure}
\hspace{2em}
\begin{subfigure}{0.65\textwidth}
\includegraphics[width=0.95\textwidth]{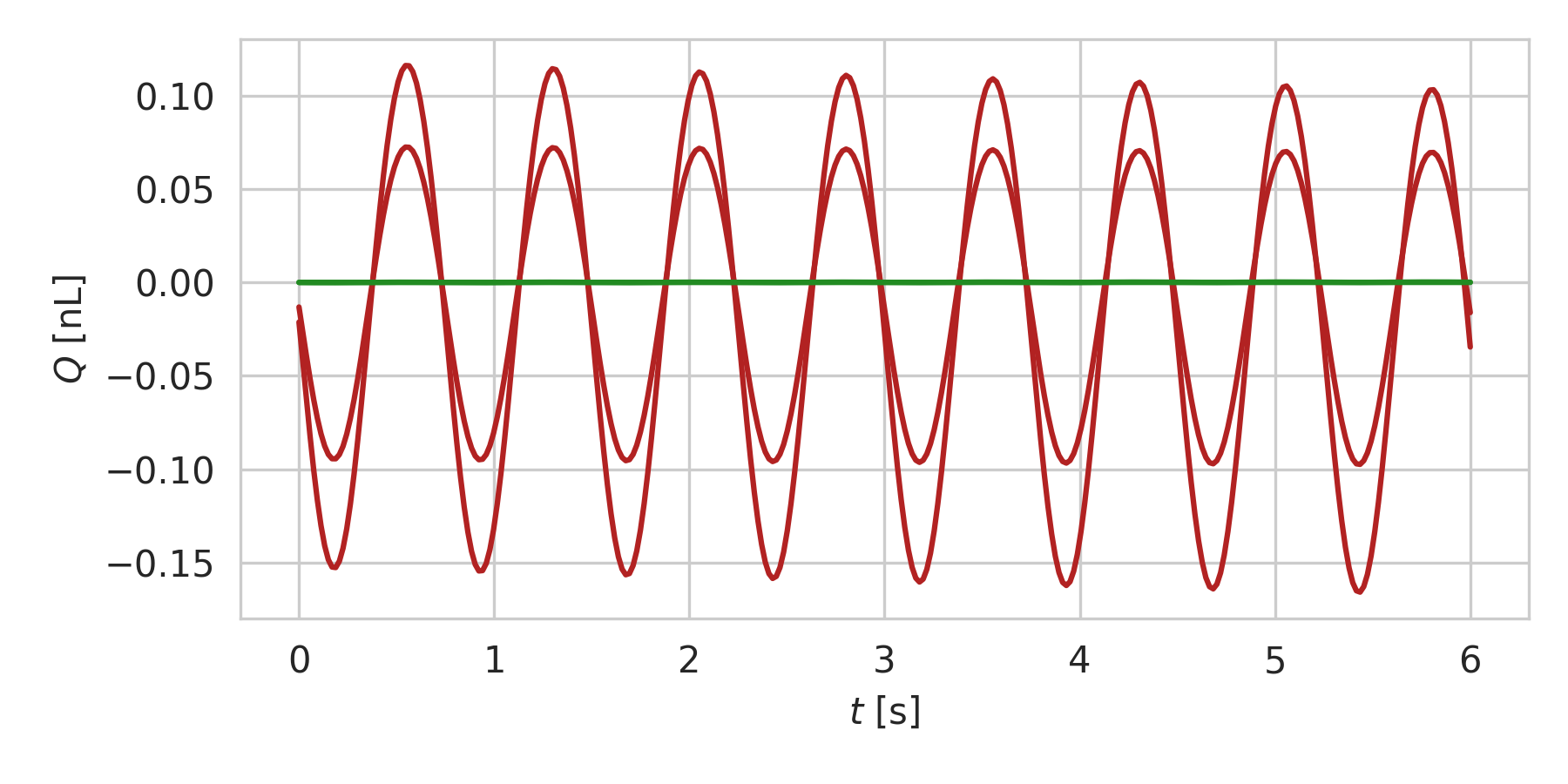}
\caption{}
\label{fig:micro-results}
\end{subfigure}
\caption{Uniform arterial pulsations create purely pulsatile flow in an image-based perivascular network extending continuously from arteries, to capillaries, to veins. The vessels are contained in a 1mm$^3$ cube of cortical tissue; arteries are marked in red, veins in blue, and capillaries in green. Cross-section fluxes were recorded at two arterial inlets and two venous outlets, marked using red and green spheres, respectively. The flow at arterial inlets was found to be purely pulsatile. Negligible flow was recorded at venous outlets, indicating that, with this configuration, the flow induced by arterial perivascular pumping is limited to arterial vessels.}
\end{figure}

\begin{table}
\begin{subtable}{0.18\textwidth}
\begin{tikzpicture}
    \node[anchor=south west, inner sep=0] (image) at (0,0) {
\includegraphics[width=0.9\textwidth]{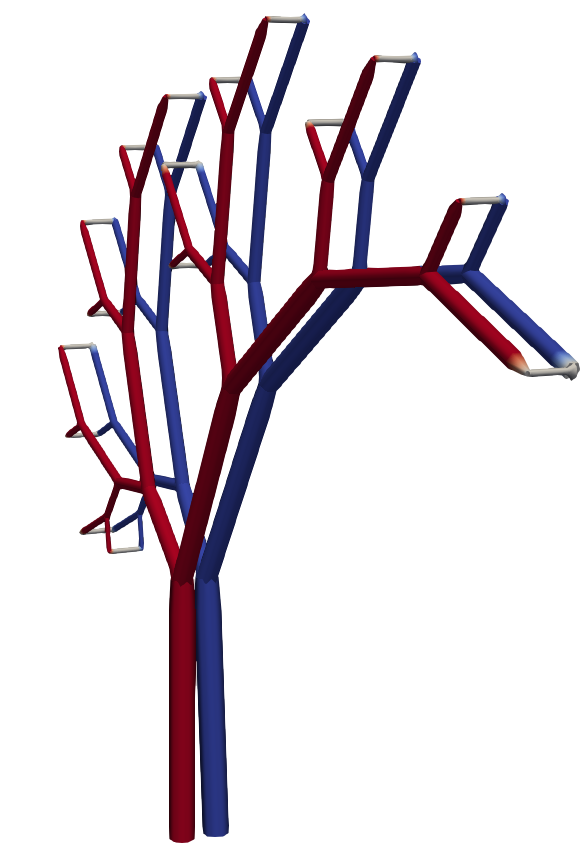} };
        \draw[black,thick] (0.75,0.1) circle (0.1cm); 
        \draw[black,thick] (0.9,0.1) circle (0.1cm); 
        \node[draw=none, text=black] at (0.2,0.1) {\scriptsize{inlet}};
        \node[draw=none, text=black] at (1.6,0.1) {\scriptsize{outlet}};
        \node[draw=none, text=red] at (-0.2,1) {\scriptsize{arteries}};
        \node[draw=none, text=blue] at (1.5,1) {\scriptsize{veins}};
        \node[draw=none, text=teal] at (1,3.5) {\scriptsize{capillaries}};
    \end{tikzpicture} 
\end{subtable}
\hspace{2em}
\begin{subtable}{0.8\textwidth}
\begin{tabular}{l|*{3}{c}}\toprule 
\backslashbox{$N_{\text{gen}}$}{$\langle Q \rangle (\eta)$}
&\makebox[3em]{$\sfrac{R_c^1}{R_{a, \text{min}}^1}=1$}&\makebox[3em]{$\sfrac{R_c^1}{R_{a, \text{min}}^1}=0.1$}&\makebox[3em]{$\sfrac{R_c^1}{R_{a, \text{min}}^1}=0.01$} \\
\midrule 
       3  &  0.2$\mu$L (0.9\%) & 0.0 $\mu$L (0.0\%)& 0.0 $\mu$L (0.0\%)  \\
       4  &  0.4$\mu$L (1.3\%) & 0.0 $\mu$L (0.0\%)& 0.0 $\mu$L (0.0\%)  \\
       5  &  0.7$\mu$L (1.8\%) & 0.0 $\mu$L (0.0\%) & 0.0 $\mu$L (0.0\%)  \\
       \bottomrule
    \end{tabular}    
    \end{subtable}
    \caption{Uniform arterial contraction/expansion waves introduce negligible net flow in a connected arterial-capillary-venous PVS. The PVS is modelled as an annular space surrounding an idealized vascular network, with one arterial root node and one venous root node. As the capillary inner radius $R_c^1$ is reduced relative to the minimum arterial radius $R_{a, \text{min}}^1$, the capillary resistance increases and net flow is disrupted. This indicates that influx and efflux routes are necessary to produce net flow.}
    \label{tab:connectome}
\end{table}

To better understand these observations, we repeat these simulations but in an idealized PVS network extending continuously from arteries to capillaries to veins. The arterial and venous sides are modelled as identical vascular trees with $N_\text{gen}$ generations, and connected via edges acting as capillaries (Figure ~\ref{tab:connectome}). To assess the impact of capillary resistance on net flow, we adjust the inner radius $R_c^1$ of the capillary vessels to be a given fraction of the minimum arterial inner radius $R_{a, \text{min}}^1$.
Table \ref{tab:connectome} reports the net flow and directionality ratio $\eta$ induced by arterial pulsations with this configuration. Indeed, the directional flow component quickly vanishes as the capillary radii shrinks, increasing the resistance of the capillary vessels.

\section{Discussion}
\label{sec:discussion}

The results of this paper are three-fold. First, we present a rigorously derived Stokes--Brinkman network model for representing fluid flow in open or porous perivascular spaces with generalized annular cross-sections. Second, we study the existence, uniqueness and stability of solutions and numerical approximations to these equations. Specifically, we prove that the approximations converge uniformly with respect to the network topology and cardinality in appropriately weighted norms. Third, by simulating CSF flow in perivascular networks, we find that uniform wave pulsations may induce directional net flow given sufficient fluid influx and efflux pathways.

The observation that spatially synchronous pulsations of the blood vessel wall (uniform waves) at the frequency of cardiac pulsations (1Hz) can induce directional net flow of CSF in the perivascular spaces is notable. Experimental observations of rapid solute transport along brain surface arteries in lockstep with the movements of the arterial walls have pointed at the presence of such perivascular flow \cite{mestre2018flow, bedussi2018paravascular}. However, these findings have been hard to reconcile with modelling based on computational fluid dynamics. Many theoretical and computational studies have found that the long wavelength of arterial wall pulsations ($\sim$100mm) compared to the shorter typical vessel length ($\leqslant$ 1 mm) does not admit net flow by perivascular pumping~\cite{asgari2016glymphatic, rey2018pulsatile, kedarasetti2020arterial, martinac2019computational, daversin2020mechanisms}. However, if the wave length and vessel length are of comparable size, then the notion that peristaltic pumping can induce non-negligible net flow is also supported by theoretical considerations~\cite{wang2011fluid, thomas2019fluid, gjerde2023directional}. Thus, vascular dynamics at lower wavelengths (and higher amplitudes) such as e.g.~stimulus-evoked or spontaneous vasomotion~\cite{van2020vasomotion, munting2023spontaneous} and their modulations during sleep \cite{bojarskaite2023sleep} can also support net flow~\cite{gjerde2023directional}. Most of these theoretical or computational studies have considered single vessel segments. Our findings indicate that the network architecture plays a significant role. This concept is in agreement with our previous observations \cite{gjerde2023directional} that net flow induced by travelling vascular waves may be amplified or diminished by nonlinear network interactions. 

Another key observation is that the connection between the PVS and tissue is vital to admit directional net flow. In the PVS network configurations where net flow is observed, the network inlets and outlets are open and thus allow for CSF flux into and out of the network. Conversely, when the PVS was modelled as a network continuously extending from the arterial to venous side with only a few influx and efflux routes, a collapse in net flow was observed. Note that anatomically, the outer layer of the PVS is covered by a mosaic of astrocytic endfeet \cite{mathiisen2010perivascular} with inter-endfeet gaps. The endfeet or their gaps may provide an alternative route for the exchange of fluid between the PVS and the surrounding tissue, as also supported by the permeability estimates of \citet{koch2023estimates}. Recent work \cite{gan2023perivascular,bork2023astrocyte} has postulated that these endfeet can act as valves, which could act as an additional mechanism driving net PVS flow.

In terms of modelling limitations, we here consider only motion of the inner perivascular wall, ignoring the elasticity of the surrounding tissue. Moreover, we do not model pressure interactions between the PVS network flux and the surrounding tissue. Both of these aspects would be expected to reduce the net flow observed within the network. In the simulations of the arterial-capillary-venous network, all vessels including capillaries and veins pulsate, which can also be considered an extreme case. 

The model and numerical methods presented here provide a robust and computationally efficient approach to simulate perivascular flow in non-trivial networks. The simulation code, built on \cite{gjerde-joss}, is openly available~\cite{thecode} and provides a solid technological foundation for further computational studies of perivascular fluid flow and transport.  


\section{Acknowledgments}
We thank James Fairbanks, Timothy Hosgood, Ridgway Scott, C\'ecile Daversin-Catty and Alexandra Vallet for their comments and input. Miroslav Kuchta gratefully acknowledges support from the Research Council of Norway, grant 303362. Marie E.~Rognes acknowledges support and funding from the Research Council of Norway (RCN) via FRIPRO grant agreement \#324239 (EMIx) and from the European Research Council (ERC) under the European Union's Horizon 2020 research and innovation programme under grant agreement 714892 (Waterscales). Barbara Wohlmuth gratefully acknowledges the financial support provided by the German Science Foundation (DFG) under project number 465242983 within the priority programme "SPP 2311: Robust coupling of continuum-biomechanical in silico models to establish active biological system models for later use in clinical applications – Co-design of modeling, numerics and usability"  (WO 671/20-1) and (WO 671/11-1).


\appendix

\section{Derivation of the reduced model}
\label{sec:derivation}

In this section, we will show how the three-dimensional Stokes equation \eqref{eq:stokes} can be reduced to the one-dimensional Stokes--Brinkman equation \eqref{eq:1d-pvs-1}-\eqref{eq:1d-pvs-2}. Let $\bar{f}=\bar{f}(s)$ denote the average of a function $f$ over the cross-section boundary $\partial C$, 
\begin{align}
\bar{f}= \frac{1}{\vert \partial C \vert} \int_{\partial C} f \dtheta,
\label{eq:avg-gamma}
\end{align}
and let $\bar{\bar{f}}=\bar{\bar{f}}(s)$ denote the average over the cross-section $C$
\begin{align}
\bar{\bar{f}} = \frac{1}{A} \int_{C} f r\dr \dtheta.
\label{eq:avg-C}
\end{align}
Recalling that the reduced model is posed in terms of the cross-section pressure and the cross-section flux, we can now write this as
\begin{align}
    p = \bar{\bar{p}}^{\text{ref}} \text{ and } q=A \bar{\bar{v}}_s^{\text{ref}}.
\end{align}
Let $R$ be a function so that $R = R^1$ on the inner boundary and $R=R^2$ on the outer boundary of $C$. To move derivatives out of the integrals, we will make use of Reynolds transport theorem, which yields the following relations:
\begin{align}
\int_{C} \pdel s f \, r \dr \dtheta &= \pdel s \bar{\bar{f}} -  \int_{\partial C} \pdel s R f \dtheta, \label{eq:reyn-transport-ds}
\\
\int_{C} \pdel t f \, r \dr \dtheta &= \pdel t \bar{\bar{f}} - \int_{\partial C} w f(s,R,\theta) \dtheta, \label{eq:reyn-transport-dt}
\end{align}
where $w=\ww \cdot \nn$.

\subsection{Reduced conservation equation}
\label{sec:redcons}
Consider a single channel with centerline $\Lambda$. In this section, we show how the conservation equation $\nabla \cdot \vv^{\text{ref}}=0$ can be reduced to a one-dimensional equation for the cross-section flow. 

Recall the Frenet-Serret frame of $\Lambda$ with $\TT, \BB, \NN$. We denote by $X(s),Y(s)$ the cylindrical coordinate system associated with the normal and binormal vectors $\BB$ and $\NN$. We decompose $\vvref=(v_r^\text{ref}, v_\theta^{\text{ref}}, v_s^{\text{ref}})$, where $v_r^\text{ref}$ and $v_\theta^{\text{ref}}$ denote the radial and angular components of $\vvref$ with respect to $X, Y$, and $v_s^{\text{ref}}$ is the component of $\vvref$ in the tangent direction $\TT$.

In the cylindrical coordinate system, the conservation equation then reads:
\begin{align*}
\partial_s v_s^{\text{ref}}  + \frac{1}{r}\pdel{\theta} v_\theta^{\text{ref}} + \frac{1}{r} \left(  \pdel{r} \left(r v_r^\text{ref} \right)\right)=0.
\end{align*}
Fixing $s \in (0, l)$, and integrating over the cross-section $C(s)$, we find
\begin{align*}
\int_{C} \partial_s v_s^{\text{ref}}   r\dr \dtheta
&= \partial_s  \int_{C} v_s^{\text{ref}}   r\dr \dtheta-  \int_{\partial C} \partial_s R \underbrace{v_s^{\text{ref}}}_{=0} \, r \dtheta = \partial_s \q, \\
\int_{C}  \frac{1}{r}\pdel{\theta} v_\theta^{\text{ref}}   r\dr \dtheta
&= \int_{R^1}^{R^2} \int_0^{2\pi} \pdel{\theta} v_\theta^{\text{ref}}   \dtheta \dr =\int_{R^1}^{R^2} \underbrace{v_\theta^{\text{ref}}\vert_{\theta=2\pi}-v_\theta^{\text{ref}}\vert_{\theta=0}}_{= 0} \dr = 0, \\
\int_{C}  \frac{1}{r} \left(  \pdel{r} \left(r v_r^\text{ref} \right)\right) r\dr
&= \int_{0}^{2\pi} \int_{R^1}^{R_2} \pdel{r} \left(r v_r^\text{ref} \right) \dr\dtheta =  \int_{\partial C } R v_r^\text{ref} \dtheta \\
&=\int_{\partial C } R w \dtheta = \vert \partial C  \vert \, \bar{w}.
\end{align*}
For the first term, we used \eqref{eq:reyn-transport-ds} and that $v_s^{\text{ref}}=0$ on $\partial C$ (due to the no-slip boundary condition on $\Gamma$). For the second term,  we used the fundamental theorem of calculus. For the third term, we used $v_r^\text{ref}=w$ on $\Gamma$.

Inserting this in the conservation of mass equation yields
\begin{gather}
\begin{aligned}
\int_{C} \left( \partial_s v_s^{\text{ref}} + \frac{1}{r} \left(  \pdel{r} \left(r v_r^\text{ref} \right)\right) \right) r\dr
&= \partial_s \q - \vert \partial C \vert \bar{w} = 0.  \label{eq:temptemp}
\end{aligned}
\end{gather}
Using that $\partial_t A = \vert \partial C \vert  \bar{w}$, the integrated conservation equation \eqref{eq:temptemp} then reads
\begin{align}
\partial_s \q+ \partial_t A=0. \label{eq:1dconseq}
\end{align}
This equation is well known in the context of blood flow models \cite{olufsen1999structured, formaggia2003one}, and simply states that changing the size of the cross-section will drive a cross-section flux. One may notice that this an exact result, meaning that the reduction holds without any assumptions on $p^{\text{ref}}$ and $v^{\text{ref}}$. 

\subsection{Reduced axial momentum equation}
Recall the assumption \eqref{eq:assumption-v}, stating that 
$v_s^\text{ref}=\hat{v}(s,t) v^{vp}(r,\theta;s,t)$, where $v^{vp}(r,\theta;s,t)$ is the velocity profile associated with the cross-section $C=C(s,t)$. In this section, we show how this assumption can be used to derive a reduced momentum equation. 

The (full) axial momentum equation reads
$v_s^{\text{ref}}$ reads
\begin{align*}
\partial_t v_s^{\text{ref}} - \frac{\nu}{\varphi} \Delta v_s^{\text{ref}} + \frac{\nu}{\kappa} v_s^{\text{ref}} + \partial_s p  &= 0.
\end{align*}
The dimension reduction is performed by integrating this equation over an arbitrary portion $\tilde{\Omega}$ of the annular cylinder, 
\begin{align*}
\int_{\tilde{\Omega}}\left( \partial_t v_s^{\text{ref}} - \frac{\nu }{\varphi} \Delta v_s^{\text{ref}}  + \frac{\nu}{\kappa} v_s^{\text{ref}}  + \partial_s p  \right)\, r \dr\dtheta\ds &= 0, \label{eq:temp-integral}
\end{align*}
and using Reynolds transport theorem to transfer the derivatives out of the integral.

We evaluate the results of \eqref{eq:temp-integral} term by term. For the first term, 
\begin{align*}
\int_{\tilde{\Omega}}   \partial_t v_s^{\text{ref}}  \, r \dr\dtheta\ds 
&= \int_{s_1}^{s_2} \int_{C}  \partial_t   v_s^{\text{ref}}  r\, \dr\dtheta \ds
\\
&= \int_{s_1}^{s_2} \left( \partial_t  \int_{C} v_s^{\text{ref}}  r\, \dr\dtheta   -  \int_{\partial C} \underbrace{v_s^{\text{ref}}}_{=0 \text{ on} \Gamma} w R d\theta \right) \ds  =  \int_{s_1}^{s_2} \partial_t  q  \ds,
\end{align*}
where we used \eqref{eq:reyn-transport-dt} to move the time derivative out of the integral.

For the second term, we apply the divergence theorem, which yields
\begin{align*}
\int_{\tilde{\Omega}} \Delta v_s^{\text{ref}}  \, r\dr\dtheta\ds &= \int_{\partial \tilde{\Omega}}  \nabla v_s^{\text{ref}} \cdot \nn \, \mathrm{d}\boldsymbol{\sigma} \\
&=  \underbrace{\int_{\mathrm{C}(s_2,t)} \partial_s v_s^{\text{ref}} r\dr\dtheta - \int_{\mathrm{C}(s_1,t)} \partial_s v_s^{\text{ref}} r\dr\dtheta}_{\text{top and bottom boundary}} \\ &+ \underbrace{\int_{s_1}^{s_2} \int_{\partial C(s,t)} \nabla v_s^{\text{ref}} \cdot \nn \dtheta \ds}_{\text{inner and outer lateral boundary}},
\end{align*}
where $\mathrm{C}(s_2,t)$ and $\mathrm{C}(s_1,t)$ denote the top and bottom boundaries of $\tilde{\Omega}$. For the top and bottom boundary terms, we compute
\begin{align*}
& \int_{C(s_2,t)} \partial_s v_s^{\text{ref}} \,r\dr\dtheta - \int_{C(s_1,t)} \partial_s v_s^{\text{ref}} \,r\dr\dtheta  \\
&= \int_0^{2\pi} \left( \int_{R^{1}(s_2,\theta,t)}^{R^{2}(s_2,\theta,t)} \partial_s v_s^{\text{ref}}(s_2,t) r \, \dr - \int_{R^{1}(s_1,\theta,t)}^{R^{2}(s_1,\theta,t)} \partial_s v_s^{\text{ref}}(s_1,t) r \, \dr\right) \dtheta  \\
&=\partial_s  \int_0^{2\pi} \left( \int_{R^{1}(s_2,\theta,t)}^{R^{2}(s_2,\theta,t)}  v_s^{\text{ref}}(s_2,t) r \, \dr-  \int_{R^{1}(s_1,\theta,t)}^{R^{2}(s_1,\theta,t)} v_s^{\text{ref}}(s_1,t) r \, \dr \right) \dtheta  \\
&=\partial_s \left( \q(s_2,t) -\q(s_1,t) \right) \\
&=\partial_s \int_{s_1}^{s_2} \partial_s( \q(s,t))\ds \\
&= \int_{s_1}^{s_2} \frac{\partial^2}{\partial s^2} \q(s,t)\ds
\end{align*}
where we moved the partial derivative $\partial_s$ directly out of the integral as the integration domains $C(s_1)$ and $C(s_2)$ do not depend on $s$.

For the lateral boundary terms, we recall the splitting \eqref{eq:assumption-v}, which implies $\nabla v_s^{\text{ref}} = (\partial_r, \partial_\theta, \partial_s) v_s^{\text{ref}} =  (\hat{v}\partial_r v^{vp}, \hat{v}\partial_\theta v^{vp}, v^{vp}\partial_s \hat{v})$. On the inner and outer lateral boundaries, the outward pointing unit normal is perpendicular to the axial direction. Thus $\nabla v_s^{\text{ref}} \cdot \nn = \nabla v^{vp} \cdot \nn$. From this we find
\begin{align*}
\int_{s_1}^{s_2} \int_{\partial C}  \nabla v_s^{\text{ref}} \cdot \nn \,\mathrm{d}\sigma \ds &= \int_{s_1}^{s_2}  \hat{v}\int_{\partial C}  \nabla v^{vp} \cdot \nn \,r\dr\dtheta \ds \\
&= \int_{s_1}^{s_2}  \hat{v}\int_{C} \Delta v^{vp} \,r\dr\dtheta \ds  \\
&= \int_{s_1}^{s_2}  \hat{v} A \overline{\overline{\Delta v^{vp}} } \ds
\end{align*}
where we used the divergence theorem again.

By a straightforward calculation, we have
\begin{align*}
\q &= \int_0^{2\pi} \int_{R^{1}}^{R^{2}} v_s^{\text{ref}} \, r \dr \dtheta= \hat{v}\int_0^{2\pi} \int_{R^{1}}^{R^{2}} v^{vp} \, r \dr \dtheta =  \hat{v} A \bar{\bar{v}}^{vp} \\
 \Rightarrow \hat{v} &= \frac{\q}{A \bar{\bar{v}}^{vp}}.
\end{align*}

Consequently
\begin{align*}
\int_{\tilde{\Omega}}  \frac{\nu}{\varphi}  \Delta v_s^{\text{ref}} \, r\dr\dtheta\ds &= \int_{s_1}^{s_2} \frac{\nu}{\varphi} \frac{\partial^2}{\partial s^2}  \q -  \frac{\nu}{\varphi} \frac{\overline{\overline{\Delta v^{vp}}} }{\bar{\bar{v}}^{vp}} \q  \ds.
\end{align*}

For the third term (the Brinkman term), we have by definition
\begin{align*}
\int_{\tilde{\Omega}}  \frac{\nu}{\kappa} v_s^{\text{ref}} \, r\dr\dtheta\ds = \int_{s_1}^{s_2} \frac{\nu}{\kappa} \q \ds.
\end{align*} 

For the fourth and final term (the pressure term), we assumed $p=p(s,t)$; thus, we simply have
\begin{align*}
\int_{\tilde{\Omega}} \partial_s p \, r\dr\dtheta\ds &=  \int_{s_1}^{s_2} \partial_s p \int_0^{2\pi} \int_{R^{1}}^{R^{2}} 1 \, r\, dr \dtheta \ds = \int_{s_1}^{s_2} \partial_s p  \ds .
\end{align*}

This yields the following integrated momentum equation:
\begin{align*}
\int_{s_1}^{s_2} \partial_t \q - \frac{\nu}{\varphi} \partial_{ss} \q  - \frac{\nu}{\varphi} \frac{\overline{\overline{\Delta v^{vp}} }}{\bar{\bar{v}}^{vp}} \q + \frac{\nu}{\kappa} \q + \partial_s p \ds =0.
\end{align*}

As this holds for arbitrary $s_1, s_2$, we have the following averaged momentum equation for $v_s^{\text{ref}}$:
\begin{align}
\pdel t \q - \frac{\nu}{\varphi} \partial_{ss} \q  + \frac{\nu}{\varphi} \frac{\overline{\overline{\Delta v^{vp}} }}{\bar{\bar{v}}^{vp}} \q  +\frac{\nu}{\kappa} \q + \partial_s p =0.
\end{align}
with
\begin{align}
\mathcal{R}(s,t) = \frac{\nu}{\varphi} \frac{\overline{\overline{\Delta v^{vp}}}}{\bar{\bar{v}}^{vp}} + \frac{\nu}{\kappa}. \label{eq:res-appendix}
\end{align}
In the next section, we show how the velocity profile $v^{vp}$ (and hence the resistance $\mathcal{R}$) can be computed for a cross-section $C$.

\subsubsection{Computing the resistance}

Consider a domain with a single, straight unit length centerline aligned with the $z$-axis and a constant cross-section $C$:
\begin{align*}
\Omega = \{ (r\cos(\theta),r\sin(\theta),s): R^1(\theta)<r<R^2(\theta), 0<s<l \}.
\end{align*}
Next, we assume the flow in this domain is independent of time and driven by some constant pressure drop. Then $p=p(s)$, and we have $\nabla p=(0,0,\partial_s p)=(0,0,-c)$. Inserting this in \eqref{eq:stokes}, we see that the flow $\vv$ is purely axial, i.e. $\vv = (0,0,v_s^{\text{ref}})$. Moreover, from conservation of mass, we have
\begin{align*}
\nabla \cdot \vv = (0,0, \partial_s v_s^{\text{ref}}) = 0 \quad \Rightarrow \quad v_s^{\text{ref}}=v_s^{\text{ref}}(r,\theta).
\end{align*}
In this case, we have the splitting $v^{\text{ref}}=\hat{v}v^{vp}(r,\theta)$, where $\hat{v}$ is a constant. 

Inserting this in the axial component of the momentum equation in the Stokes--Brinkman system \eqref{eq:stokes}, we find
\begin{align}
 - \frac{\nu}{\varphi} \Delta v_s^{\text{ref}} + \frac{\nu}{\kappa}v_s^{\text{ref}} = \frac{c}{\hat{v}}. \label{eq:split}
\end{align}
Notice that the scaling $\hat{v}$ in \eqref{eq:assumption-v}
is arbitrary; we now fix it so that $c/(\nu \hat{v})=1$. The velocity profile associated with a cross-section $C$ can then be obtained by solving
\begin{gather}
\begin{aligned}
- \frac{1}{\varphi} \Delta v^{vp} + \frac{1}{\kappa} v^{vp} &= -1 && \text{ in } C, \\
v^{vp} &= 0 && \text{ on } \partial C. \label{eq:vp-on-C}
\end{aligned}
\end{gather}
Averaging the first equation in \eqref{eq:vp-on-C} yields
\begin{align*}
  \overline{\overline{\Delta v^{vp}}} &= \varphi \left( 1+\frac{1}{\kappa} \overline{\overline{v^{vp}}}\right).
\end{align*}
Inserting this in \eqref{eq:res-appendix}, we find
\begin{align}
\mathcal{R}(s,t) = \frac{\nu}{\q^{vp}} + 2 \frac{\nu}{\kappa},
\end{align}
where $\q^{vp}=\bar{\bar{v}}^{vp}$ is the velocity profile cross-section flux.

\begin{remark}
For open channels ($\kappa\rightarrow \infty$), this result agrees with the one derived in  \cite{tithof2019hydraulic}. Specifically, they non-dimensionalize the Stokes equations to derive the resistance 
\begin{align}
    \mathcal{R} = \frac{\nu}{q^{vp}} \frac{1}{(R^1)^4 }.
\end{align}
\end{remark}

\subsubsection{Reduced boundary and bifurcation conditions}

To derive the reduced boundary conditions, we simply average the traction boundary condition \eqref{eq:boundary-condition} over the cross-section. This yields
\begin{align}
\frac{1}{A} \int_{C} ( \frac{\nu}{\varphi} \partial_s v_s^{\text{ref}}-p^{\text{ref}})\, r \dr = \frac{1}{A} \left( \frac{\nu}{\phi} \q - p\right) =  \tilde{p}^{\text{ref}}.
\end{align}

\end{document}